\documentclass[11pt]{amsart}
\usepackage{amsmath}%
\usepackage{amsaddr}
\usepackage[T1]{fontenc}
\usepackage{amsthm}
\usepackage{amsxtra}%
\usepackage{amsfonts}%
\usepackage{amssymb}%
\usepackage{color,hyperref}
\usepackage{graphicx}
\usepackage{cite}
\usepackage{subfigure}
\hypersetup{colorlinks,breaklinks,
             linkcolor=blue,urlcolor=blue,
           anchorcolor=blue,citecolor=blue}

\usepackage[margin=1.25in]{geometry}

\renewcommand{\O}{{\mathcal O}}
\renewcommand{\bar}[1]{{\overline{#1}}}

\renewcommand{\P}{{\mathbb P}}
\newcommand{\X}{{\mathcal X}}
\newcommand{\G}{{\mathcal G}}
\newcommand{\W}{{\mathcal W}}
\newcommand{\R}{\mathbb{R}}
\newcommand{\T}{\mathbb{T}}
\renewcommand{\div}{\text{div}}
\newcommand{\N}{\mathbb{N}}

\newcommand{\E}{\mathbb{E}}

\newcommand{\Z}{\mathbb{Z}}

\renewcommand{\phi}{\varphi}

\newtheorem{theorem}{Theorem}
\newtheorem{lemma}[theorem]{Lemma}
\newtheorem{corollary}[theorem]{Corollary}

\newtheorem{proposition}[theorem]{Proposition}
\theoremstyle{definition}
\newtheorem{remark}[theorem]{Remark}
\newtheorem{definition}[theorem]{Definition}

\begin{document} 
\title[Consistency of Lipschitz learning]{Consistency of Lipschitz learning with infinite unlabeled data and finite labeled data}
\author{Jeff Calder}
\address{Department of Mathematics, University of Minnesota}
\email{jcalder@umn.edu}


\begin{abstract}
We study the consistency of Lipschitz learning on graphs in the limit of infinite unlabeled data and finite labeled data. Previous work has conjectured that Lipschitz learning is well-posed in this limit, but is insensitive to the distribution of the unlabeled data, which is undesirable for semi-supervised learning. We first prove that this conjecture is true in the special case of a random geometric graph model with kernel-based weights. Then we go on to show that on a random geometric graph with \emph{self-tuning} weights, Lipschitz learning is in fact highly sensitive to the distribution of the unlabeled data, and we show how the degree of sensitivity can be adjusted by tuning the weights. In both cases, our results follow from showing that the sequence of learned functions converges to the viscosity solution of an $\infty$-Laplace type equation, and studying the structure of the limiting equation. 
\end{abstract}

\maketitle

\section{Introduction}

In many machine learning problems, such as website classification or medical image analysis, an expert is required to label data, which may be costly, while the cost of acquiring unlabeled data can be negligible in comparison. This discrepancy has led to the development of learning algorithms that make use of not only the labeled data, but also properties of the unlabeled data in the learning task. Such algorithms are called \emph{semi-supervised} learning \cite{ssl}, as opposed to \emph{fully supervised} (uses only labeled data) or \emph{unsupervised} (uses no label information). A large class of semi-supervised learning algorithms are \emph{graph-based}, where the data is given the structure of a graph with similarities between data points, and the task is to deduce some interesting information about data in certain regions of the graph.

Let us describe a general formulation of graph-based semi-supervised learning. Let $\G=(\X,\W)$ be a weighted graph with vertices $\X$ and nonnegative edge weights  $\W=\{w(x,y)\}_{x,y\in \X}$. Assume we are given a label function  $g:\O\to \R$ where  $\O\subset \X$ are the labeled vertices. The graph-based semi-supervised learning problem is to extend the labels from  $\O$ to the remaining vertices of the graph  $\X\setminus \O$. The problem is not well-posed as stated, since there is no unique way to extend the labels. One generally makes the \emph{semi-supervised smoothness assumption}, which says that the learned labels must vary smoothly through dense regions of the graph. 
 
There are many ways to impose the semi-supervised smoothness assumption, one of the most popular and successful being Laplacian regularization \cite{zhu2003semi}, which corresponds to the optimization problem
\[\min_{u:\X\to \R}\sum_{x,y\in \X}w(x,y)^2( u(x) - u(y) )^2\ \ \text{ subject to } u(x) = g(x) \text{ for all } x \in \O. \]
It has recently been observed \cite{el2016asymptotic,nadler2009semi}  that Laplacian regularization is ill-posed in the limit of infinite unlabeled and finite labeled data. The label function $u$ degenerates into a constant label that is some type of average of the given labels. In other words, the learned function forgets about the labeled data. In \cite{el2016asymptotic}, the authors study the $p$-Laplacian regularization
\begin{equation}\label{eq:Lplearn}
\min_{u:\X\to \R}\sum_{x,y\in \X}w(x,y)^p|u(x) - u(y)|^p\ \ \text{ subject to } u(x) = g(x) \text{ for all } x \in \O 
\end{equation}
as a replacement for Laplacian regularization in the setting of few labels.
Taking (formally) $p\to \infty$ above one obtains \emph{Lipschitz learning}, which was proposed earlier in \cite{kyng2015algorithms,luxburg2004distance}. Lipschitz learning amounts to solving the problem
\begin{equation}\label{eq:Lp}
\min_{u:\X\to \R} \max_{x,y\in \X} \{w(x,y)|u(x) - u(y)|\}\ \ \text{ subject to } u(x) = g(x) \text{ for all } x \in \O. 
\end{equation}
We mention the Lipschitz learning problem \eqref{eq:Lp} does not in general have a unique solution. Roughly speaking, one can modify any minimizer $u(x)$ away from any pair $(x,y)$ that maximizes the gradient $Du(x,y):=w(x,y)|u(x)-u(y)|$ to obtain another (in fact, an infinite family) of minimizers. To fix this issue, one normally considers the unique minimizer whose gradient $Du(x,y)$ as an element of $\R^{|\X|^2}$ is smallest in the \emph{lexicographical ordering}~\cite{kyng2015algorithms}.\footnote{For two vectors $x,y\in \R^n$ that are ordered $x_1\leq x_2\leq \cdots \leq x_n$ and $y_1\leq y_2\leq \cdots \leq y_n$, we say $x\leq y$ in the \emph{lexicographical ordering} if $x_i< y_i$ at the first entry $i$ where $x_i\neq y_i$. To apply the lexicographical ordering to unordered vectors, we simply order the components of each vector from least to greatest and then apply the ordering.} This amounts to minimizing the largest gradient, and the second largest, and third largest, and so on.

The authors of \cite{el2016asymptotic} were motivated by the Lipschitz learning problem, but were unable to address it directly and instead studied the $p$-Laplace problem \eqref{eq:Lplearn} for large $p$. They showed that for random geometric graphs
\begin{equation}\label{eq:limJp}
\lim_{h\to 0^+}\lim_{n\to \infty}\frac{1}{n^2h^{d+p}}\sum_{x,y\in \X}w(x,y)^p|u(x) - u(y)|^p = C\int \rho^2 |\nabla u|^p\, dx=:J_p(u),
\end{equation}
where $n$ is the number of vertices in the graph, $\rho$ is the data density, and $u$ is a smooth function on $\R^d$. From this result, one can conjecture that solutions of the $p$-Laplace learning problem in the continuum have gradients with bounded $L^p$-norm (i.e., $\int |\nabla u|^p\, dx<\infty$), and by the Sobolev embedding theorem \cite{EvansPDE} are H\"older continuous for $p>d$. This suggests the $L^p$-learning problem is well-posed in the limit of infinite unlabeled and finite labeled data when $p>d$. The authors of \cite{el2016asymptotic} also point out that the Euler-Lagrange equation satisfied by minimizers of $J_p(u)$, defined in \eqref{eq:limJp}, appears to forget about the distribution $\rho$ of the unlabeled data as $p\to \infty$. This suggests that Lipschitz learning ($p=\infty$) is insensitive to the distribution of the unlabeled data. Our initial goal in this work was to formulate and prove this conjecture rigorously. In the course of this work, we discovered that the insensitivity to unlabeled data is a more subtle point, and crucially depends on how one selects the weights in the graph. In particular, for a particular choice of \emph{self-tuning} weights, Lipschitz learning can be made highly sensitive to the distribution $\rho$.

Let us mention that while the formal consistency result \eqref{eq:limJp}, proved in \cite{el2016asymptotic}, is suggestive, it is not sufficient to prove solutions of the graph problem converge in the continuum limit to the solution of a continuum variational problem or partial differential equation. This was addressed in follow-up works for the variational $p$-Laplacian by Slep\v cev and Thrope \cite{slepvcev2017analysis} using $\Gamma$-convergence tools, and for the game-theoretic $p$-Laplacian by Calder \cite{calder2018game} using the theory of viscosity solutions. In particular, in \cite{slepvcev2017analysis}, it was shown that one cannot take the limit as $n\to \infty$ first, and then $h\to 0^+$ afterwards, as is done in \eqref{eq:limJp}, otherwise the problem becomes again ill-posed (e.g., the solution forgets the labeled data) even for $p>d$. In fact, there is a length scale restriction $h\ll (1/n)^{1/p}$ identified in \cite{slepvcev2017analysis}, where $h$ is the bandwidth of the kernel used to define the weights (see \eqref{eq:weights_sigma}), which necessitates sending $n\to \infty$ and $h\to 0$ simultaneously.

The learning problem \eqref{eq:Lplearn} is closely related to the graph $p$-Laplacian. Indeed, we can differentiate the energy in  \eqref{eq:Lplearn} to see that any minimizer satisfies the graph $p$-Laplace equation 
\begin{equation}\label{eq:graphpLap}
\sum_{y\in \X}w(x,y)^p|u(y)-u(x)|^{p-2}(u(y)-u(x)) = 0 \ \ \ \text{ for all } x\in \X\setminus \O,
\end{equation}
subject to the Dirichlet condition $u=g$ on $\O$. Deriving the Euler-Lagrange equation for Lipschitz learning \eqref{eq:Lp} is less direct, since we seek the lexicographic minimizer. To deduce the Euler-Lagrange equation for Lipschitz learning, let us consider sending $p\to \infty$ in \eqref{eq:graphpLap}. To do this, we separate the positive and negative terms, writing \eqref{eq:graphpLap} as
\[\sum_{\substack{y\in \X \\ u(y) > u(x)}}w(x,y)^p(u(y)-u(x))^{p-1}=\sum_{\substack{y\in \X \\ u(y) \leq u(x)}}w(x,y)^p(u(x)-u(y))^{p-1}.\]
We note that both sides must have at least one term in the sum, unless $u$ is constant at all neighbors.
Since the terms in the sums on both sides are non-negative, we can take the $p^{\rm th}$ root of both sides and send $p\to \infty$ to obtain
\[\max_{y \in \X}w(x,y)(u(y)-u(x)) = \max_{y \in \X}w(x,y)(u(x)-u(y)).\]
Rearranging we get the graph $\infty$-Laplace equation
\begin{equation}\label{eq:LpEL}
\max_{y \in \X}w(x,y)(u(y)-u(x))  + \min_{y \in \X}w(x,y)(u(y)-u(x)) = 0.
\end{equation}
While this argument is formal, it can be made rigorous without much trouble, showing that solutions of the graph $p$-Laplace equation converge to solutions of the graph $\infty$-Laplace equation as $p\to \infty$. It is also possible to derive the $\infty$-Laplace equation \eqref{eq:LpEL} directly from the lexicographic minimization property, which is done in \cite{kyng2015algorithms}. 

The graph $p$-Laplace and $\infty$-Laplace equations are closely connected to their continuum counterparts in the theory of partial differential equations (PDE) \cite{lindqvist2017notes}. The continuum version of \eqref{eq:Lplearn} is the variational problem
\begin{equation}\label{eq:pDir}
\min_u \int_{\Omega} |\nabla u|^p \, dx.
\end{equation} 
The Euler-Lagrange equation satisfied by minimizers of \eqref{eq:pDir} is the $p$-Laplace equation 
\begin{equation}\label{eq:pLapEq}
\Delta_pu:=\text{div}(|\nabla u|^{p-2}\nabla u) = 0.
\end{equation}
Solutions of \eqref{eq:pLapEq} are called $p$-harmonic, and arise in problems such as nonlinear potential theory \cite{lindqvist2017notes} and stochastic tug-of-war games \cite{peres2009tug,peres2008tug,lewicka2014game}, among many other applications. We note that the divergence can be formally expanded to show that any $p$-harmonic function also satisfies
\begin{equation}\label{eq:plapexpand}
|\nabla u|^{p-2}(\Delta u + (p-2)\Delta_\infty u) = 0,
\end{equation}
where $\Delta_\infty$ is the $\infty$-Laplace operator defined for $\nabla  u\neq 0$ by
\[\Delta_\infty u := \frac{1}{|\nabla u|^{2}} \sum_{i,j=1}^d u_{x_ix_j}u_{x_i}u_{x_j}.\]
We can divide \eqref{eq:plapexpand}  by $p|\nabla u|^{p-2}$  to see that any $p$-harmonic function satisfies 
\[\tfrac{1}{p}\Delta u  + \left( 1-\tfrac{2}{p} \right)\Delta_\infty u = 0.\]
Sending $p\to \infty$ we obtain the $\infty$-Laplace equation $\Delta_\infty u = 0$, which justifies the notation. It is possible to show that solutions of $\Delta_p u =0$ converge to solutions of $\Delta_\infty u=0$ as $p\to \infty$ (see, e.g., \cite{aronsson2004tour}), however, the reader should be cautioned that the same is not true for solutions of $\Delta_p u = f$ for nonzero $f$, since the step where we cancelled the term $|\nabla u|^{p-2}$ is no longer valid (see \cite{ishii2005limits}). 

We mention it is also possible to send $p\to \infty$ in the variational problem \eqref{eq:pDir}, provided one is careful about interpreting the limit. The formal limit problem $\min_u \|\nabla u\|_{L^\infty}$ does not have unique solutions for the same reason as in the graph-based case; near any point where $|\nabla u|$ is less than the supremum, we are free to modify $u$ without changing the objective function. To resolve this in the continuum setting, one looks for minimizers that are \emph{absolutely minimal} \cite{aronsson2004tour}. A Lipschitz function $u:\Omega\to \R$ is absolutely minimal if 
\[u = v \text{ on } \partial V \implies \|Du\|_{L^\infty(V)}\leq \|Dv\|_{L^\infty(V)},\]
for each $V\subset \Omega$ open and bounded and each $v\in C(\bar{V})$. In other words, $u$ is absolutely minimal if its Lipschitz constant cannot be locally improved. It turns out that the property of being absolutely minimal is equivalent to solving the $\infty$-Laplace equation $\Delta_\infty u=0$ in the viscosity sense \cite{aronsson2004tour}. This variational interpretation of the $\infty$-Laplacian is the prototypical example of a calculus of variations problem in $L^\infty$ \cite{barron2001euler}.

In this paper, we rigorously study the consistency of Lipschitz learning in the limit where the fraction of labeled points is vanishingly small, that is, we take the limit of infinite unlabeled data and finite labeled data. We prove that Lipschitz learning is well-posed in this limit, and that the learned functions converge to the solution of an $\infty$-Laplace type equation, depending on the choice of weights in the graph. For the standard choice of weights $w_{xy} = \Phi(\tfrac{|x-y|}{h})$ in a random geometric graph, the limiting $\infty$-Laplace equation $\Delta_\infty u = 0$ does not depend on the distribution of the unlabeled data, which means that Lipschitz learning is fully-supervised, and not semi-supervised in this limit, as was conjectured in \cite{el2016asymptotic}. However, for a graph with self-tuning weights (see Eq.~\eqref{eq:weights}), which are common in machine learning, we show that the limiting $\infty$-Laplace equation does depend on, and can be highly sensitive to, the distribution of the unlabeled data. In particular, the PDE includes a first order drift term that propagates labels along the negative gradient of the distribution. Thus, the observed insensitivity to the data distribution is a merely a function of the choice of weights in the graph, and is not inherent in Lipschitz learning. This suggests that self-tuning weights may be important in Lipschitz learning. We also present the results of numerical simulations on synthetic and real data showing that self-tuning weights improve classification accuracy for Lipschitz learning with very few labels. 

We mention that, contrary to most consistency results on graph Laplacians (e.g., \cite{hein2007graph}), our results make minimal use of probability and do not depend on the \emph{i.i.d}~assumption. In fact, our first result (Theorem \ref{thm:pinf}) on standard Lipschitz learning does not use probability at all, and simply requires the data to densely fill out a domain. Our second result (Theorem \ref{thm:pinf2}) on Lipschitz learning with self-tuning weights, requires that a kernel density estimator for the data density converges to a smooth function. This holds for random data in both \emph{i.i.d.}~ and non-\emph{i.i.d.}~settings (see Remark \ref{rem:noniid2} for a non-\emph{i.i.d.}~example). The reason the proof can work in non-\emph{i.i.d.}~settings is that the graph $\infty$-Laplacian involves the max and min of a sequence of random variables, instead of a sum, and the max and min can be bounded by controlling the size of the largest ``hole'' in the data, and do not require concentration of measure results, for which the \emph{i.i.d.}~assumption is crucial. We describe our results in more detail below.

\section{Main results}
\label{sec:main}

Here, we describe the setup and our main results. We mention that we use the analysis convention that $C,c>0$ denote arbitrary constants, whose value may change from line to line. We work on the flat Torus $\T^d=\R^d/\Z^d$, that is, we take periodic boundary conditions. For each  $n\in \N$ let  $X_n\subset \T^d$ be a collection of  $n$ points.   Let $\O \subset \T^d$ be a fixed finite collection of points and set
\begin{equation}\label{eq:graph}
\X_n := X_{n}\cup \O.
\end{equation}
The points $\X_n$ will form the vertices of our graph. To select the edge weights, let $\Phi:[0,\infty)\to [0,\infty)$ be a $C^2$  function satisfying 
\begin{equation}\label{eq:phi}
\begin{cases}
\Phi (s)\geq 1,&\text{if }s\in (0,1)\\
\Phi (s)=0,&\text{if }s\geq 2.
\end{cases}
\end{equation}
Select a length scale $h_n>0$ and define the weights
\begin{equation}\label{eq:weights_sigma}
\sigma_n(x,y) := \Phi\left( \frac{|x -y|}{h_n} \right),
\end{equation} 
where $|x-y|$ denotes the distance on the torus.
This choice of weights is standard in the construction of a random geometric graph and is widely used in consistency results \cite{hein2007graph,trillos2016continuum}. We now modify the construction to include self-tuning weights, which is standard in learning problems (see, e.g., \cite{ting2011analysis}). Given a constant $\alpha\in\R$, we define the self-tuning weights
\begin{equation}\label{eq:weights}
w_n(x,y) := d_n(x)^\alpha d_n(y)^\alpha \sigma_n(x,y),
\end{equation} 
where $d_n(x)$ is the (normalized) degree of vertex $x$ given by
\begin{equation}\label{eq:deg}
d_n(x) = \frac{1}{nh_n^d}\sum_{y\in \X_n}\sigma_n(x,y).
\end{equation}
Let  $\W_n=\left\{ w_n(x,y) \right\}_{x,y\in \X_n}$ and let   $\G_n=(\X_n,\W_n)$ be the graph with vertices  $\X_n$ and edge weights  $\W_n$.  We note that when $\alpha=0$ we get the standard construction of a random geometric graph, while for $\alpha>0$ the weights are larger in denser regions of the graph.

Let  $g:\O\to \R$ and let  $u_n:\X_n\to \R$ be the solution of the Lipschitz learning problem \eqref{eq:Lp}. As we discussed above (and will prove in Section \ref{sec:maxprinc}), the function  $u_n$ satisfies the optimality conditions 
\begin{equation}\label{eq:optinf}
\left\{\begin{aligned}
L_{n}u_n &=0&&\text{in } X_n\\
u_n&=g&&\text{in } \O,\end{aligned}\right.
\end{equation}
where $L_{n}$ is the graph $\infty$-Laplacian defined by
\begin{equation}\label{eq:dJi}
L_nu(x) := \max_{y \in \X_n} w_n(x,y)(u(y) -u(x)) + \min_{y \in \X_n} w_n(x,y) (u(y) -u(x)).
\end{equation}
We also define 
\begin{equation}\label{eq:r}
r_n=\sup_{x\in \T^d}\text{dist}(x,\X_n).
\end{equation}
We note that the graph  $\G_n$ is connected  whenever $r_n<h_n/(4\sqrt{d})$.  The only assumption we place on the data $\X_n$ at the moment is that $r_n\to 0$ fast enough so that
\begin{equation}\label{eq:rn}
\lim_{n\to \infty}\frac{r_n^2}{h_n^3}=0.
\end{equation}
This ensures, in particular, that the graph is connected as $n\to \infty$.

We first present a result for the standard random geometric graph with $\alpha=0$.
\begin{theorem}
Suppose that $\alpha=0$, $h_n \to 0$ and $r_n \to 0$ as  $n\to \infty$ so that \eqref{eq:rn} holds. Then
\begin{equation}
u_n \longrightarrow u \ \ \text{ uniformly on }\T^d \text{ as } n \to\infty,
\label{eq:convinf2}
\end{equation}
where $u \in C^{0,1}(\T^d)$ is the unique viscosity solution of the $\infty$-Laplace equation
\begin{equation}
\left\{
\begin{aligned}
\Delta_\infty u &= 0&&\text{in } \T^d\setminus \O\\
u&=g&&\text{on } \O.\\
\end{aligned}
 \right.
\label{eq:infLap}
\end{equation}
\label{thm:pinf}
\end{theorem}
We remark that Theorem \ref{thm:pinf} is a generalization (to the random graph setting) of the convergence results of Oberman \cite{oberman2013finite,oberman2005convergent} for a similar scheme for the $\infty$-Laplace equation on a uniform grid. We note that the viscosity solution of \eqref{eq:infLap} is in general only Lipschitz continuous, and is not  a classical $C^2$  solution. The notion of viscosity solution is based on the maximum principle, and is the natural notion of weak solution for nonlinear elliptic equations. Viscosity solutions are only required to be continuous functions, and satisfy the partial differential equation in a weak sense. We define viscosity solution for \eqref{eq:infLap} in Section \ref{sec:proof}.  For more details on viscosity solutions, we refer the reader to the user's guide \cite{crandall1992user}. 

\begin{remark}\label{rem:noniid}
Notice that Theorem \ref{thm:pinf} makes no assumptions on the distribution of the unlabeled data $X_n$. The unlabeled data may be deterministic or random, and if random, may not be \emph{i.i.d.} This says that while Lipschitz learning on standard random geometric graphs is well-posed in the limit of infinite unlabeled and finite labeled data, the limit is completely independent of the unlabeled data, and so the algorithm is fully supervised, and not semi-supervised, in this limit. This was suggested by the authors of \cite{el2016asymptotic}, and Theorem \ref{thm:pinf} provides a rigorous statement of this result. 
\end{remark}

We now consider the case where $\alpha\neq 0$. We assume there exists a function $f\in C^2(\T^d)$ such that for 
\begin{equation}\label{eq:Rn}
R_n := \sup_{x\in \X_n}\{|d_n(x) - f(x)|\}
\end{equation}
we have
\begin{equation}\label{eq:Rnlimit}
\lim_{n\to \infty}\frac{R_n}{h_n} = 0.
\end{equation}
In other words, the asymptotic expansion
\begin{equation}\label{eq:Rnlimit2}
d_n(x) = f(x) + o(h_n)
\end{equation}
holds uniformly in $x$ as $n\to \infty$.  Here, $d_n(x)$ is, up to a constant, a kernel density estimator \cite{silverman2018density} for the density of both labeled and unlabeled data. Since the number of labeled data points is finite as $n\to \infty$, the function $f(x)$ represents the density of the \emph{unlabeled data}. Remark \ref{rem:iid} makes this precise.  We note that, in practice, it is not necessary to use the same bandwidth $h_n$ for the kernel density estimator $d_n$ and the weights $\sigma_n$, and there may be situations where decoupling these quantities is advantageous.
\begin{remark}\label{rem:knngraphs}
For $k$-nearest neighbor graphs, the degree \eqref{eq:deg} is not an estimate of the data distribution, that is, \eqref{eq:Rnlimit2} does not hold. Indeed, in an unweighted $k$-nearest neighbor graph the degree is constant. In this case, we can slightly modify the self-tuning weights to use a $k$-nearest neighbor density estimator. Letting $D_{n,k}(x)$ denote the distance from $x$ to the $k^{\rm th}$ nearest neighbor in $\X_n$, self-tuning weights for a $k$-nearest neighbor graph can be defined as 
\begin{equation}\label{eq:selfknn}
w_n(x,y) := D_{n,k}(x)^{-\alpha} D_{n,k}(y)^{-\alpha} \sigma_n(x,y).
\end{equation}
We expect the conclusions of Theorem \ref{thm:pinf2} below to hold in this setting with minor modifications.
\end{remark}

To derive the continuum PDE when $\alpha\neq 0$, we note that the continuum variational problem to \eqref{eq:Lp} is $\min_u \|f^{2\alpha}\nabla u\|_{L^\infty}$. We can approximate this problem by a sequence of $p$-Laplace type problems of the form
\[\min_u \int f^{2\alpha p}|\nabla u|^p\, dx\]
as $p\to \infty$. The Euler-Lagrange equation  for this problem is
\[\div(f^{2\alpha p}|\nabla u|^{p-2}\nabla u)=0.\]
Expanding the divergence we obtain
\[f^{2\alpha p}|\nabla u|^{p-2}(\Delta u + 2\alpha p \nabla \log f \cdot \nabla u + (p-2)\Delta_\infty u) = 0.\]
Cancelling the term out front and sending $p\to \infty$ we formally obtain the $\infty$-Laplace equation
\[\Delta_\infty u + 2\alpha \nabla \log f \cdot \nabla u = 0.\]

We now present our result for $\alpha\neq 0$, which verifies the formal arguments above.
\begin{theorem}
Suppose that $\alpha\neq 0$, and that $h_n \to 0$, $r_n \to 0$, and $R_n\to 0$ as  $n\to \infty$ so that \eqref{eq:rn} and \eqref{eq:Rnlimit} hold.
Then 
\begin{equation}
u_n \longrightarrow u \ \ \text{ uniformly on }\T^d \text{ as } n \to\infty,
\label{eq:convinf}
\end{equation}
where $u \in C^{0,1}(\T^d)$ is the unique viscosity solution of the $\infty$-Laplace type equation
\begin{equation}
\left\{
\begin{aligned}
\Delta_\infty u + 2\alpha \nabla \log f \cdot \nabla u &= 0&&\text{in } \T^d\setminus \O\\
u&=g&&\text{on } \O.\\
\end{aligned}
 \right.
\label{eq:infLap2}
\end{equation}
\label{thm:pinf2}
\end{theorem}

Several remarks are in order.

\begin{remark}\label{rem:dist}
Notice the continuum PDE \eqref{eq:infLap2} in Theorem \ref{thm:pinf2} contains the additional linear term $\nabla \log f \cdot \nabla u$. This is a drift (also called advection or transport) term that acts to propagate the labels along the negative gradient of $\log f$. Since $f$ represents the distribution of the unlabeled data, this additional drift term acts to propagate labels from regions of high density to regions of lower density (when $\alpha>0$; the reverse is true for $\alpha<0$). Hence, Lipschitz learning with self-tuning weights is not only well-posed in the limit of infinite unlabeled data and finite labeled data, but the algorithm also remembers the structure of the unlabeled data, and the degree of sensitivity to unlabeled data can be controlled by tuning the parameter $\alpha$. Hence, Lipschitz learning with self-tuning weights retains the benefits of semi-supervised learning in the limit of infinite unlabeled data, which suggests that self-tuning weights are very important in applications of Lipschitz learning with few labels.
\end{remark}

\begin{remark}\label{rem:classicallearning}
In contrast with classical learning theory \cite{bousquet2004introduction}, the regularity of the label function $g$ does not play a role in this setting of finite labeled and infinite unlabeled data, because we are very coarsely sampling $g$. Instead, the regularity of the solution $u$ of the limiting partial differential equation \eqref{eq:infLap2} is important in controlling rates of convergence in Theorems \ref{thm:pinf} and \ref{thm:pinf2}. Unfortunately, $u$ is a viscosity solution, and is at best Lipschitz continuous, so it is impossible to exploit regularity of $u$ to prove convergence rates, as is often done in classical learning theory. There may be other techniques available to prove convergence rates (see, e.g., \cite{smart2010infinity}), and we leave this to future work.
\end{remark}
\begin{remark}\label{rem:iid}
We can specialize Theorems \ref{thm:pinf} and \ref{thm:pinf2} to the case where $X_n=\left\{ Y_1,Y_2,Y_3,\dots,Y_n \right\}$ is a sequence of independent and identically distributed  random variables with a $C^2$ probability density function $\rho$ bounded away from zero (strictly positive).  The two conditions we need to verify are \eqref{eq:rn} and \eqref{eq:Rnlimit}.

Obtaining the condition \eqref{eq:rn} is standard in probability; we include the brief argument here for the reader's convenience. We partition $\T^d$ into $t^{-d}$ cubes $B_1,B_2,B_3,\dots$ of side length $t>0$. If $r_n \geq \delta \sqrt{d}$ then at least one cube must contain no points from the sample $X_n$, and so 
\[\P(r_n \geq t \sqrt{d}) \leq \sum_{i=1}^{t^{-d}} \P(X_n\cap B_i=\varnothing)\leq t^{-d}(1- \gamma \,t^d)^n,\]
where $\gamma := \min_{\T^d}\rho > 0$, and we assume $t$ is small enough so that $\gamma\,t^d < 1$. Using $\log(1+x)\leq x$ with $x=-\gamma\, t^d$ we obtain
\[\P(r_n \geq t \sqrt{d}) \leq \exp\left( -\gamma n t^d - d\log(t) \right).\]
Setting $t^2 = \delta h_n^3/d$ yields
\[\P(r_n^2/h_n^3\geq \delta)\leq \exp\left(  -\gamma d^{-d/2} nh_n^{3d/2}\delta^{d/2} -\tfrac{d}{2} \log(\delta h_n^3/d) \right),\]
and hence \eqref{eq:rn} holds almost surely provided
\begin{equation}\label{eq:hs_inf}
\lim_{n\to \infty}\frac{nh_n^{3d/2}}{\log(n)} = \infty.
\end{equation}

The condition \eqref{eq:Rnlimit} follow from standard kernel density estimation theory. Indeed, notice we can write
\[d_n(x) = \frac{1}{nh_n^d}\sum_{y\in X_n}\sigma_n(x,y) + O\left( \frac{|\O|h_n^2}{nh_n^{d+2}} \right).\]
Since $X_n$ is a sequence of \emph{i.i.d.}~random variables,  it is a standard fact in the kernel density estimation literature (see Appendix \ref{app:iid}) that \eqref{eq:Rnlimit} holds provided
\begin{equation}\label{eq:hs_inf2}
\lim_{n\to \infty}\frac{nh_n^{d+2}}{\log(n)} = \infty.
\end{equation}
In summary, in the \emph{i.i.d.}~case, the condition \eqref{eq:rn} in Theorem \ref{thm:pinf} can be replaced with \eqref{eq:hs_inf}, while in Theorem \ref{thm:pinf2} we require both \eqref{eq:hs_inf} and \eqref{eq:hs_inf2} to hold.

The reader should contrast this with the requirement that 
\[\lim_{n\to \infty}\frac{nh_n^{d}}{\log(n)} = \infty \quad \text{ or }\quad \lim_{n\to \infty}\frac{nh_n^{d +2}}{\log(n)} = \infty\]
for the consistency results in Laplacian based regularization \cite{trillos2016continuum,hein2005graphs}. The reason for the difference is that for the graph Laplacian, one needs to control the fluctuations in a sum of random variables, and typically the Bernstein inequality is used for this. For the $\infty$-Laplacian, which involves the maximum of a collection of random variables, the techniques to establish concentration are significantly different.
\end{remark}

\begin{remark}\label{rem:noniid2}
We note that Theorem \ref{thm:pinf2} does not require the \emph{i.i.d.}~assumption. We simply need the kernel density estimator \eqref{eq:deg} to be consistent, i.e., \eqref{eq:Rnlimit} must hold for some $f$. There are many examples of non-\emph{i.i.d.}~data for which kernel density estimators are consistent in this sense. For example, the data may be deterministic, and then \eqref{eq:Rnlimit} holds if the data is sufficiently uniformly spread out. A deterministic example is a grid $h\Z^d$.

For a more involved and realistic example, we can consider data of the form $X_n = \{\tau(Y_i,Y_j)\}_{i\neq j}$, where $Y_1,\dots,Y_m$ is a sequence of \emph{i.i.d.}~random variables with $C^2$ density $\rho$, $1\leq i,j\leq m$, and $n=m(m-1)$. Our dataset $X_n$ is thus a collection of $n$ identically distributed, but \emph{not independent}, random variables. 
 Data in this form arises in problems in statistical analysis of spatial point patterns \cite{illian2008statistical,frees1994estimating}, and in claims models for insurance dealing with sums of insurance claims \cite{frees1994estimating}, among many other problems. As an example, if $Y_i$ and $Y_j$ represent spatial positions, then $\tau(Y_i,Y_j)$ could represent any notion of distance between $Y_i$ and $Y_j$, which is called the \emph{interpoint distance}. There are many problems, such as prediction of airline flight delays \cite{choi2016prediction,rebollo2014characterization}, where labels are assigned to origin-destination pairs in this way.
   
In this setting, the condition \eqref{eq:Rnlimit} is essentially the problem of density estimation for functions of observations, such as $U$-statistics, which is the focus of much work in statistics (see, e.g.,  \cite{gine2007local}).  The condition \eqref{eq:Rnlimit} holds with
\begin{equation}\label{eq:marginal}
f(z) = \int_{\R^d}\Phi(|x|)\, dx\int_{\R^d}\rho(\psi_y(z)) |D_x \tau(\psi_y(z),y)|^{-1} \rho(y) dy,
\end{equation}
provided $\tau$ satisfies some non-degeneracy conditions, and
\begin{equation}\label{eq:hs_inf4}
\lim_{n\to \infty}\frac{\sqrt{n}h_n^{d+2}}{\log(n)} = \infty.
\end{equation}
In \eqref{eq:marginal}, $\psi_y$ is the inverse of $x\mapsto \tau(x,y)$.
We review a proof of these facts, and make precise our assumptions on $\tau$, in Appendix \ref{app:noniid}. We note this construction can easily be generalized to higher degrees of dependence, such as $X_n = \{\tau(Y_i,Y_j,Y_k)\}_{i\neq j\neq k}$, and so on. In these cases, the rate \eqref{eq:hs_inf4} worsens in the dependence on $n$ (e.g., $n^{1/3}, n^{1/4}$, etc.), since there is less independence in the data.
\end{remark}

\begin{remark}
We note that the requirement  $\Phi \in C^2$ in Theorems \ref{thm:pinf} and \ref{thm:pinf2}  is  necessary;  it is used in the proof of Lemma \ref{lem:approximation1}. If  $\Phi \in C^1$  then the proof of Theorem \ref{thm:pinf} can be modified by using \eqref{eq:approximation2} in place of \eqref{eq:approximation1} from Lemma \ref{lem:approximation1}. The only difference is that the condition \eqref{eq:rn} must be replaced with 
\[\lim_{n\to \infty}\frac{r_n}{h_n^3}=0.\]
Remark   \ref{rem:iid} remains true provided \eqref{eq:hs_inf} is replaced by 
\[\lim_{n\to \infty}\frac{nh_n^{3d}}{\log(n)} = \infty.\]
\end{remark}
\begin{remark}
If instead of working on the Torus  $\T^d$, we take our unlabeled points to be sampled from a domain  $X_n\subset \Omega\subset \R^d$, then we expect that Theorem \ref{thm:pinf} will hold  under similar hypotheses with the additional boundary condition
\[\frac{\partial u}{\partial \nu}=0\quad \text{  on }\partial \Omega.\] 
\end{remark}

\subsection{Outline}

The rest of the paper is organized as follows. In Section \ref{sec:maxprinc} we discuss the maximum principle for the graph  $\infty$-Laplacian and prove existence and uniqueness of  solutions to \eqref{eq:optinf}. In Section \ref{sec:cons} we prove consistency of the graph $\infty$-Laplacian for graphs with self-tuning weights, and in Section \ref{sec:proof} we review the definition of viscosity solution, and then give the proofs of Theorems \ref{thm:pinf} and \ref{thm:pinf2}. We conclude in Section \ref{sec:conc}.

\section{The maximum principle}
\label{sec:maxprinc}

In this section we show that \eqref{eq:optinf} is well-posed and establish \emph{a priori} estimates on the solution  $u_n$. The proof relies on the maximum principle on a graph, which we review below. 
 
We first introduce some notation. We say that  $y$ is adjacent to  $x$ whenever  $w_n(x,y)>0$.  We say that the graph  $\G_n=(\X_n,\W_n)$  is connected to $\O\subset \X_n$ if for every  $x\in \X_n\setminus \O$  there exists $y\in \O$  and a path from  $x$ to $y$ consisting of adjacent vertices.    

We now present the maximum principle for the graph $\infty$-Laplace equation. 
\begin{theorem}[Maximum principle]\label{thm:maxprinc}
Assume the graph $\G_n=(\X_n,\W_n)$ is connected to $\O\subset \X_n$. Let $u,v:\X_n \to \R$ satisfy
\[L_nu(x)\geq 0\geq L_nv(x).\]
Then 
\begin{equation}\label{eq:maxprinc}
\max_{\X_n}(u - v) = \max_{\O} (u - v).
\end{equation}
\end{theorem}
The proof of Theorem \ref{thm:maxprinc} in a similar setting was proved in \cite{manfredi2015nonlinear}. We include a simple proof here for completeness. 
\begin{proof}
Define
\[M=\{x \in \X_n \, : \, u(x)-v(x) = \max_{\X_n} (u-v)\}.\]
If  $M\cap \O\neq  \varnothing $ then we are done, so we may assume that  $M\cap \O=\varnothing $.   Let $x \in M$. Then we have 
\[u(x) -u(y)\geq v(x) -v(y)\quad \text{for all }y\in \X_n.\]
It follows that  $L_nu(x)\leq L_nv(x)$.  The opposite inequality is true by hypothesis, and hence $L_nu(x)=0=L_nv(x)$ whenever  $x\in M$.  This implies that 
\[A:=\max_{y \in \X_n} w_n(x,y) (u(x)-u(y)) = \max_{y \in \X_n} w_n(x,y) (v(x)-v(y)),\]
\[B:=\min_{y \in \X_n} w_n(x,y) (u(x)-u(y)) = \min_{y \in \X_n} w_n(x,y) (v(x)-v(y)),\]
and $A + B = -L_n u(x)=-L_n v(x)=0$. We now have two cases.

1. If $A>0$ and $B<0$ then there exists $z \in \X_n$ such that
\[w_n(x,z)(u(x)-u(z)) = \min_{y\in \X_n} w_n(x,y)(u(x) - u(y)) < 0.\]
Therefore $u(z) > u(x)$ and
\[u(z) - v(z) = -(u(x) - u(z)) + v(x) - v(z) + u(x) - v(x)\geq u(x)-v(x).\]
It follows that
\begin{equation}\label{eq:a}
u(z) - v(z) = u(x)-v(x),\  u(z) > u(x), \text{ and } v(z) > v(x).
\end{equation}

2. If $A=B= 0$ then $u(x)=u(y)$ and $v(x)=v(y)$ for all $y$ adjacent to $x$, and so  
\[u(y) - v(y) = u(x) - v(x)\]
for all $y$ adjacent to $x$.

Let $Q_1 \subset M$ be the collection of points for which case 1 holds, and let $Q_2\subset M $ be the points for which case 2 holds. We construct a path in $M$ inductively as follows. Let $x_0 \in M$ and suppose we have chosen $x_0,\dots,x_k$. If $x_k \in Q_1$, we choose $x_{k+1}=z$ as in case 1 above. If $x_k \in Q_2$, then we find a path $x_k=y_1,\dots,y_\ell$ from $x_k$ to $y_\ell\in\O$. Let
\[j = \max\{i \, : \,  y_q \in Q_2 \text{ for all } 1 \leq q \leq i\}.\]
Since $y_j \in Q_2$, case 2 holds and so we have $y_{j+1} \in M$. Therefore $j+1\leq \ell-1$, $y_{j+1} \not\in \O$, and  $y_{j +1}\in Q_1$. Choose $x_{k+1}=z$ as in case 1. We terminate the construction when $x_{k+1} \in \O$.

This constructs a path $x_0,x_1,\dots,x_k,\dots$ belonging to $M$ such that $u-v$ is constant along the path, and $u$ is strictly increasing, i.e., 
\[u(x_0) < u(x_1) < \cdots < u(x_k) < \cdots.\]
Therefore, the path cannot revisit any point, and must eventually terminate at some $x_T \in \O$. Since $u-v$ is constant along the path, we have
\[\max_{\X_n} (u-v) = u(x_0) - v(x_0) = u(x_T) - v(x_T) \leq \max_\O (u-v),\]
which completes the proof.
\end{proof}
\begin{corollary}\label{cor:stability}
Assume the graph $\G_n=(\X_n,\W_n)$ is connected to $\O$. Let $u,v:\X_n \to \R$ satisfy 
\[L_n v(x)=0=L_n u(x) \ \text{ for all } \ x \in X_n.\]
Then
\begin{equation}\label{eq:stability}
\max_{\X_n}|u-v| = \max_{\O} |u-v|.
\end{equation}
\end{corollary}
\begin{remark}\label{rem:unique}
Corollary \ref{cor:stability} shows that \eqref{eq:optinf}  has at most one solution, and the solution is stable under perturbations in the boundary conditions.
\end{remark}

Existence of a solution to \eqref{eq:optinf} was proved in~\cite{kyng2015algorithms,sheffield2012vector} as the absolutely minimal Lipschitz extension on a graph. It is also possible to prove existence via the Perron method, as was done in \cite{calder2018game} for the game theoretic $p$-Laplacian on a graph. We record these standard existence results in the following theorem. 
\begin{theorem}\label{thm:existunique}
Assume the graph $\G_n=(\X_n,\W_n)$ is connected to $\O$. Then there  exists a unique solution  $u_n:\X_n\to \R$ of \eqref{eq:optinf}.  Furthermore, there exists a constant  $C>0$ depending only on  $\O$ and $g$ such that 
\begin{equation}\label{eq:prior}
\min_\O g \leq u_n \leq \max_\O g,\text{ and}
\end{equation}
\begin{equation}\label{eq:lipest}
\max_{x,y\in \X_n}w_n(x,y)|u_n(x) -u_n(y)|\leq Ch_n.
\end{equation}
\end{theorem}
\begin{proof}
Existence of a solution follows from~\cite{calder2018game,kyng2015algorithms,sheffield2012vector} and uniqueness follows from Corollary \ref{cor:stability}. The \emph{a priori} estimate \eqref{eq:prior} follows from Theorem \ref{thm:maxprinc}. All that is left to prove is \eqref{eq:lipest}.  Let  $\varphi \in C^1(\T^d)$  such that $\varphi (x)=g(x)$  for all $x\in \O$. Since  $u_n$  is the absolutely minimal Lipschitz extension of $g$  to the graph $\G_n$, we have that 
\[\max_{x,y\in \X_n}w_n(x,y)|u_n(x) -u_n(y)|\leq \max_{x,y\in \X_n}w_n(x,y)|\varphi (x) -\varphi (y)|\leq C\|\nabla  \varphi \|_{L^\infty(\T^d)}h_n.\]      
\end{proof}

\section{Consistency for smooth functions}
\label{sec:cons}

In this section we prove consistency for the graph $\infty$-Laplacian for smooth functions. Even though the viscosity solutions of the $\infty$-Laplace equations \eqref{eq:infLap} and \eqref{eq:infLap2} are not smooth, the viscosity solution framework allows for checking consistency only with smooth functions. 

We define the nonlocal operator 
\begin{align}\label{eq:T}
H_nu(x) &:=  \max_{y \in \T^d}\left\{ f(x)^\alpha f(y)^\alpha\sigma_n(x,y)(u(y) -u(x))  \right\}\\
&\hspace{1in}+ \min_{y \in \T^d}\left\{  f(x)^\alpha f(y)^\alpha \sigma_n(x,y) (u(y) -u(x))\right\}.\notag
\end{align}
The proof of consistency is split into two steps. First, in Lemma \ref{lem:approximation1} we show that  $L_n$  can be approximated by  $H_n$.  Then, in Theorem \ref{thm:consistency} we prove that  $H_n$ is consistent with the $\infty$-Laplace operator  $\Delta _\infty$ in the limit as  $n\to \infty$.      

\begin{lemma}\label{lem:approximation1}
Let  $\varphi \in C^2(\R^d)$.  Then 
\begin{equation}\label{eq:approximation1}
|L_n\varphi (x) -H_n\varphi (x)|\leq C\left( \|\nabla  \phi\|_\infty +h_n \|\nabla  ^2\phi\|_\infty \right) (r^2_nh_n^{ -1} + |\alpha| R_n h_n),
\end{equation}
and
\begin{equation}\label{eq:approximation2}
|L_n\varphi (x) -H_n\varphi (x)|\leq C\|\nabla  \phi\|_\infty (r_nh_n^{ -1}  + |\alpha| R_n h_n).
\end{equation}
\end{lemma}
\begin{proof}
Set
\[\psi(y) = f(x)^\alpha f(y)^\alpha\Phi\left(\frac{|x-y|}{h_n}\right)(\varphi (y) -\varphi (x)),\]
and note that
\[\|\nabla  ^2\psi\|_\infty \leq C\left( \frac{1}{h_n}\|\nabla  \phi\|_\infty + \|\nabla  ^2\phi\|_\infty \right).\]
Let $y_0\in \T^d$ be a point at which $\psi$ attains its maximum value. Since $\psi$ is $C^2$ and $\nabla  \psi(y_0)=0$ we have
\[\psi(y) \geq \psi(y_0)  - \frac{1}{2}\|\nabla  ^2\psi\|_\infty |y-y_0|^2\]
for all $y$. Set  $t=\psi(y_0) - \max_{\X_n}\psi $. Then for any  $y\in \X_n$    we have
\[\frac{1}{2}\|\nabla  ^2\psi\|_\infty |y-y_0|^2\geq \psi (y_0) -\psi (y) \geq t.\]
It follows that 
\[r^2_n\geq \frac{ 2t}{\|\nabla  ^2\psi\|_\infty },\]
and hence 
\[\max_{\T^d}\psi - \max_{\X_n}\psi\leq \frac{1}{2}\|\nabla  ^2\psi\|_\infty  r^2_n. \]
A similar argument shows that 
\[\min_{\X_n}\psi-\min_{\T^d}\psi  \leq \frac{1}{2}\|\nabla  ^2\psi\|_\infty  r^2_n. \]
Since $|d_n(x)^\alpha - f(x)^\alpha|\leq  C|\alpha| R_n$ this yields
\begin{align*}
|L_n\varphi (x) - H_n\varphi(x)|&\leq \min_{\X_n}\psi-\min_{\T^d}\psi-(\max_{\T^d}\psi - \max_{\X_n} \psi) + C\|\nabla \varphi\|_\infty |\alpha|R_n h_n \\
&\leq C(\|\nabla \phi\|_\infty + h_n\|\nabla ^2\phi\|_\infty)(r_n^2 h_n^{-1} + |\alpha| R_n h_n).
\end{align*}
This completes the proof of \eqref{eq:approximation1}.
The proof of \eqref{eq:approximation2} is similar. 
\end{proof}

Before proceeding, we need an elementary proposition.
\begin{proposition}\label{prop:max}
For any $p\in\R^d$ with $p\neq 0$ and $A\in \R^{d\times d}$ we have
\begin{equation}\label{eq:maxTE}
\left|\max_{|v|=r}\left\{ p\cdot v + \tfrac{1}{2}v\cdot A v  \right\}  - r|p| - \tfrac{1}{2}r^2|p|^{-2}p\cdot Ap\right| \leq 2r^3\|A\|^2|p|^{-1}
\end{equation}
for all $r>0$.
\end{proposition}
\begin{proof}
Let $C_r=\max_{|v|=r}\left\{ p\cdot v + \tfrac{1}{2}v\cdot A v  \right\} $. Then there exists $v_r$ with $|v_r|=r$ such that
\begin{equation}\label{eq:maxob}
C_r = p\cdot v_r + \frac{1}{2}v_r\cdot Av_r.
\end{equation}
Let $w_r := rp/|p|$. Then we have
\begin{align}\label{eq:upperbd}
C_r&=p\cdot v_r + \frac{1}{2}w_r \cdot Aw_r + \frac{1}{2}(v_r\cdot A v_r - w_r\cdot A w_r)\\
&=p\cdot v_r + \frac{1}{2}w_r \cdot Aw_r + \frac{1}{2}( (v_r-w_r)\cdot A v_r + w_r\cdot A(v_r - w_r))\notag\\
&\leq p\cdot v_r +\frac{1}{2}w_r \cdot Aw_r + \frac{1}{2}(|w_r-v_r|\|A\| |v_r| + |w_r|\|A\| |w_r-v_r|)\notag\\
&\leq  p\cdot v_r +\frac{1}{2}w_r \cdot Aw_r + r\|A\||w_r-v_r|.\notag
\end{align}
Substituting $w_r$ into the definition of $C_r$ we also have
\begin{equation}\label{eq:lowerbd}
C_r \geq r|p| + \frac{1}{2}w_r\cdot A w_r.
\end{equation}
Combining this with \eqref{eq:upperbd} we have
\begin{equation}\label{eq:pdotv}
r|p| - p\cdot v_r \leq r\|A\| |w_r-v_r|.
\end{equation}
It follows that
\begin{align*}
|w_r - v_r|^2&=|w_r|^2 - 2w_r\cdot v_r + |v_r|^2\\
&=2(r^2 - w_r\cdot v_r)\\
&=2r|p|^{-1}(r|p| - p \cdot v_r)\\
&\leq 2r^2\|A\||p|^{-1}|w_r-v_r|,
\end{align*}
and hence
\begin{equation}\label{eq:wrvr}
|w_r-v_r| \leq 2r^2 \|A\| |p|^{-1}.
\end{equation}
Recalling \eqref{eq:upperbd} we have
\[C_r \leq r|p| + \frac{1}{2}w_r\cdot A w_r + 2r^3\|A\|^2|p|^{-1}.\]
Combining this with \eqref{eq:lowerbd} completes the proof.
\end{proof}

We now prove consistency.
\begin{theorem}\label{thm:consistency}
Let $\phi \in C^3(\R^d)$ and $x_0\in \T^d$. If $\nabla  \varphi (x_0)\neq  0$ and
\begin{equation}\label{eq:negop}
\Delta_\infty \varphi(x_0) + 2\alpha \nabla \log f(x_0) \cdot \nabla \varphi(x_0) < 0,
\end{equation}
then for any sequence $x_n \to x_0$ we have
\begin{equation}\label{eq:consistencyalpha}
\limsup_{n \to \infty} \frac{1}{h_n^2}H_n \phi(x_n) < 0.
\end{equation}
\end{theorem}
\begin{proof}
Let $x\in \T^d$ and define
\begin{equation}\label{eq:Bdef}
B_n(x) = \max_{y \in \T^d}\left\{ f(x)^\alpha f(y)^\alpha \Phi(h_n^{ -1}|x-y|) (\phi(y) - \phi(x))\right\}.
\end{equation}
By Taylor expansion we have
\[f(y)^\alpha =f(x)^\alpha + \alpha f(x)^{\alpha-1} \nabla f(x)\cdot (y-x) + O(|x-y|^2).\]
Since the supremum in \eqref{eq:Bdef} is attained for $|x-y|\leq 2h_n$ we have
\begin{align*}
B_n(x)&=f(x)^{2\alpha}\max_{y\in B(x,2h_n)}\left\{ \Phi(h_n^{ -1}|x-y|)(1 + \alpha \nabla \log f(x)\cdot (y-x)) (\varphi(y)-\varphi(x)) \right\} \\
&\hspace{4in}+ O(h_n^3).
\end{align*}
Setting $v=y-x$ and continuing to Taylor expand yields
\begin{align*}
\frac{B_n(x)}{f(x)^{2\alpha}}&= \max_{v\in B(0,2h_n)}\Big\{ \Phi(h_n^{-1}|v|)\Big(\nabla \varphi(x)\cdot v + \tfrac{1}{2}v\cdot \nabla ^2\varphi(x)v \\
&\hspace{1.75in}+ \alpha (\nabla \log f(x)\cdot v)(\nabla \varphi(x)\cdot v)\Big)\Big\} + O(h_n^3)\\
& = \max_{v\in B(0,2h_n)}\left\{ \Phi(h_n^{-1}|z|)\left( p\cdot v + \tfrac{1}{2}v\cdot Av  \right) \right\} + O(h_n^3),
\end{align*}
where
\[p = \nabla \varphi(x) \ \ \text{ and } \ \ A =\nabla^2 \varphi(x) + 2\alpha \nabla \log f(x)\otimes \nabla \varphi(x).\]
By Proposition \ref{prop:max} 
\begin{align*}
\frac{B_n(x)}{f(x)^{2\alpha}}&= \max_{0 \leq r \leq 2h_n}\left\{ \Phi(h_n^{-1}r)\max_{|v|=r}\left\{ p\cdot v + \tfrac{1}{2}v\cdot Av  \right\} \right\} + O(h_n^3)\\
&= \max_{0 \leq r \leq 2h_n}\left\{ \Phi(h_n^{-1}r)\left( |p|r+ \tfrac{r^2}{2|p|^{2}}p\cdot Ap\right) \right\} + O(h_n^3)\notag\\
&= \max_{0 \leq r \leq 2h_n}\left\{ \Phi(h_n^{-1}r)\left( |\nabla \varphi|r+ \tfrac{1}{2}\Delta_\infty \varphi r^2 + \alpha\nabla \log f\cdot \nabla \varphi r^2\right) \right\} + O(h_n^3)\notag\\
&= \max_{0 \leq s \leq 2}\left\{ s\Phi(s)|\nabla \varphi|h_n+ \tfrac{1}{2}s^2\Phi(s)\left(\Delta_\infty \varphi + 2\alpha\nabla \log f\cdot \nabla \varphi\right)h_n^2 \right\} + O(h_n^3).\notag
\end{align*}
Similarly, for
\[B_n'(x) := \min_{y \in \T^d}\left\{ f(x)^\alpha f(y)^\alpha \Phi(h_n^{ -1}|x-y|) (\phi(y) - \phi(x))\right\},\]
we have
\[\frac{B_n'(x)}{f(x)^{2\alpha}}= \min_{0 \leq s \leq 2}\left\{ -s\Phi(s)|\nabla \varphi|h_n+ \tfrac{1}{2}s^2\Phi(s)\left(\Delta_\infty \varphi + 2\alpha\nabla \log f\cdot \nabla \varphi\right)h_n^2 \right\} + O(h_n^3).\]

Now, let $s_n\in [0,2]$ such that
\begin{align*}
B_n(x_n) &= s_n\Phi(s_n)|\nabla \varphi(x_n)|h_n\\
&\hspace{0.5in}+ \tfrac{1}{2}s_n^2\Phi(s_n)\Big(\Delta_\infty \varphi(x_n) +2\alpha\nabla \log f(x_n)\cdot \nabla \varphi(x_n)\Big)h_n^2+ O(h_n^3).
\end{align*}
Then  we have
\begin{align}\label{eq:almost}
H_n\varphi (x_n)&= B_n(x_n) +B_n'(x_n)\\
&\leq f(x_n)^{2\alpha}s_n^2\Phi (s_n)(\Delta _\infty\varphi (x_n) + 2\alpha\nabla \log f(x_n)\cdot \nabla \varphi(x_n))h_n^2 +Ch_n^3.\notag
\end{align}
Let $c_0 = \max_{0 \leq s \leq 2} s\Phi(s)>0$ and set $\delta:=|\nabla \varphi(x_0)|$, recalling that $\delta>0$ by assumption. We may assume $x_n$ is close enough to $x_0$ so that $|\nabla \varphi(x_n)|>\frac{\delta}{2}$. Then we have
\[\frac{1}{2}\delta c_0 h_n - Ch_n^2 \leq B_n(x_n) \leq C(s_nh_n + h_n^2).\]
It follows that there exists a constant $c>0$ such that $s_n \geq 2c-h_n$, and so $s_n\geq c>0$ for $n$ sufficiently large. Combining this with \eqref{eq:almost} completes the proof.
\end{proof}

\section{Proof of main results}
\label{sec:proof}
In this section we prove our main results, Theorems \ref{thm:pinf} and \ref{thm:pinf2}. The first step is to prove a Lipschitz  estimate on the sequence  $u_n$ (Lemma \ref{lem:lip}), which gives us compactness. Then we introduce the notion of viscosity solution for the $\infty$-Laplace equation, and complete the proof of Theorems \ref{thm:pinf} and \ref{thm:pinf2}. 
    
\begin{lemma}\label{lem:lip}
There exists $C>0$ such that whenever $r_n\leq h_n/(8\sqrt{d})$ and $R_n \leq \tfrac{1}{2}\inf_{\T^d}f$ we have 
\begin{equation}\label{eq:lip}
|u_n(x) -u_n(y)|\leq C(|x -y| +h_n)\quad \text{ for all } x,y\in \X_n.
\end{equation}
\end{lemma}
\begin{proof}
By Theorem \ref{thm:existunique} there exists  $C>0$  such that 
\begin{equation}\label{eq:liplocal}
w_n(x,y)|u_n(x) -u_n(y)|\leq Ch_n\quad \text{ for all } x,y\in \X_n.
\end{equation}
Recall that (see Eq.~\eqref{eq:weights})
\[w_n(x,y) = d_n(x)^\alpha d_n(y)^\alpha \sigma_n(x,y)\]
where $\sigma_n(x,y)\geq 1$ whenever  $|x -y|\leq h_n$.  By \eqref{eq:Rn} we have $|d_n(x)- f(x)|\leq  R_n$, and so if $R_n \leq \tfrac{1}{2}\inf_{\T^d}f$ we have
\[d_n(x)^\alpha d_n(y)^\alpha \geq \frac{1}{4^{|\alpha|}}\inf_{\T^d}f^{2\alpha}.\]
Therefore
\[w_n(x,y) \geq \frac{1}{4^{|\alpha|}}\inf_{\T^d}f^{2\alpha} \ \ \ \text{whenever} \ \ |x-y|\leq h_n.\]
Combining this with \eqref{eq:liplocal} we have
\begin{equation}\label{eq:b}
|u_n(x) -u_n(y)|\leq Ch_n\quad \text{ for all } x,y\in \X_n\text{  such that }|x -y|\leq h_n. 
\end{equation}
Partition  $\R^d$ into cubes of side lengths  $h_n/2\sqrt{d}$. Assume $r_n \leq h_n/8\sqrt{d} < h_n/4\sqrt{d}$. Then every cube must have at least one point from  $\X_n$. Therefore, for any  $x,y\in \X_n$ there exists a path  $x=x_1,x_2,x_3,\dots,x_\ell =y$ with  $x_i\in \X_n$ and $|x_i -x_{i +1}|\leq h_n$ for all  $i$ and  
\[\ell \leq d\left(   \frac{2\sqrt{d}|x -y|}{h_n} +1\right).\]
Therefore
\[|u_n(x) -u_n(y)|\leq \sum_{i=1}^{\ell-1} |u_n(x_i) -u_n(x_{i  +1})|\leq C\ell h_n.\]
\end{proof}

We recall the definition of viscosity solution for the partial differential equation 
\begin{equation}\label{eq:pde}
\Delta _\infty u + b\cdot \nabla u =0\ \ \text{ in } \Omega,
\end{equation}
 where $\Omega\subset \R^d$  is open, and $b:\Omega \to \R^d$.
\begin{definition}\label{def:viscosity}
We say that  $u\in C(\Omega)$ is a viscosity subsolution of   \eqref{eq:pde} if for every  $x_0\in \Omega$ and $\varphi \in C^\infty(\Omega)$ such that  $u -\varphi $  has a strict local maximum at $x_0$ and $\nabla  \varphi (x_0)\neq  0$ we have 
\[\Delta _\infty \phi(x_0)+ b\cdot \nabla \phi(x_0) \geq 0.\]      
We say that  $u\in C(\Omega)$ is a viscosity supersolution of \eqref{eq:pde} if $ -u$ is a viscosity subsolution of \eqref{eq:pde}.    We say that  $u\in C(\Omega)$ is a viscosity solution of \eqref{eq:pde} if  $u$ is both a viscosity sub- and supersolution of \eqref{eq:pde}. 
\end{definition}
Let  $\pi :\R^d\to \T^d$  be the projection operator.
\begin{definition}\label{def:viscosity2}
A function  $u\in C(\T^d)$  is a viscosity solution of \eqref{eq:infLap2} if $u=g$  on $\O$ and $v(x):=u(\pi (x))$ is a viscosity solution of \eqref{eq:pde} with      
\[\Omega:=\pi ^{ -1}(\T^d\setminus \O) \ \text{ and } \ b(x) = 2\alpha \nabla  \log f(x).\] 
\end{definition}
Uniqueness of viscosity solutions of \eqref{eq:infLap2} follows form the original work of Jensen \cite{jensen1993uniqueness} on uniqueness of viscosity solutions to the $\infty$-Laplace equation. In particular, we refer to Juutinen~\cite{juutinen1998minimization} for an adaptation of Jensen's argument to equations of the form \eqref{eq:infLap2}, which include the spatially dependent term $\nabla \log f\cdot \nabla u$. 

We now give the proof of of Theorems \ref{thm:pinf} and \ref{thm:pinf2}. The proofs are similar, and can be combined together.
\begin{proof}[Proof of Theorems \ref{thm:pinf} and \ref{thm:pinf2}]
Let $p_n:\T^d \to \X_n$ be the closest point projection. That is
\[|x-p_n(x)| = \min_{y\in \X_n}|x -y|.\]
Define  $v_n:\T^d\to \R$ by  $v_n(x)=u_n(p_n(x))$.   
Since  $|x -p_n(x)|\leq r_n$, it follows from Lemma \ref{lem:lip} that for any  $x,y\in \T^d$  
\begin{align*}
|v_n(x) -v_n(y)|&\leq C(|p_n(x) -p_n(y)| +h_n)\\
&=C(|p_n(x) -x +y -p_n(y) +x -y| +h_n)\\
&\leq C(|x -y| +2r_n +h_n).
\end{align*}
Since  $r_n,h_n\to 0$ as  $n\to \infty$, we can use    a variant of the Arzel\`a-Ascoli Theorem (see the appendix in \cite{calder2015PDE}) to show that there exists a subsequence, which we again  denote by  $v_n$, and a Lipschitz continuous function  $u\in C^{0,1}(\T^d)$ such that  $v_n\to u$ uniformly on  $\T^d$ as  $n\to \infty$. Since  $u_n(x)=v_n(x)$ for all  $x\in \X_n$ we have      
\begin{equation}\label{eq:uniform}
\lim_{n\to \infty}\max_{x\in \X_n}|u_n(x) - u(x)|=0.  
\end{equation}
We claim that  $u$ is the unique viscosity solution of \eqref{eq:infLap}. Once this is verified, we can apply the same argument to any subsequence of  $u_n$ to show that the entire sequence converges uniformly to $u$.   

We first show that  $u$ is a viscosity subsolution  of \eqref{eq:infLap}. Let $x_0\in \T^d$ and $\varphi \in C^\infty(\R^d)$  such that  $u -\varphi $  has a strict global maximum at the point $x_0$ and $\nabla  \varphi (x_0)\neq  0$. We need to show that
\[\Delta_\infty \varphi(x_0) + 2\alpha \nabla \log f(x_0)\cdot \varphi(x_0) \geq 0.\]
Assume, by way of contradiction, that 
\[\Delta_\infty \varphi(x_0) + 2\alpha \nabla \log f(x_0)\cdot \varphi(x_0) < 0.\]
By \eqref{eq:uniform} there exists a sequence of points  $x_n\in \X_n$   such that $u_n -\varphi $   attains its global maximum at $x_n$ and  $x_n\to x_0$ as $n\to \infty$.          Therefore 
\[u_n(x_n) -u_n(x)\geq \varphi (x_n) -\varphi (x)\quad \text{  for all }x\in \X_n.\] 
Since  $x_0 \not\in \O$, we have that  $x_n \not\in \O$ for  $n$ sufficiently large.
By Lemma \ref{lem:approximation1} 
\[0=L_nu_n(x_n)\leq L_n\varphi (x_n)\leq H_n\varphi (x_n) +C(r_n^2h_n^{ -1} + |\alpha| R_n h_n)\quad \text{ for all }n\geq 1.\]
By Theorem \ref{thm:consistency} and the assumption  $r_n^2h_n^{ -3}\to 0$ and (if $\alpha\neq 0$) $R_nh_n^{-1} \to 0$ as $n\to \infty$ 
\[0\leq \limsup_{n\to \infty}\frac{1}{h_n^2}H_n\varphi (x_n) +C(r^2_nh_n^{ -3} + |\alpha| R_nh_n^{-1})< 0.\]   
This is a contradiction. Thus $u$ is a viscosity subsolution of \eqref{eq:infLap}. 

To verify the supersolution property, we simply set  $v_n= -u_n$  and note that $L_nv_n= -L_nu_n=0$ and  $v_n\to  -u$  uniformly as $n\to \infty$.   The argument given above for the subsolution property shows that  $ -u$ is a viscosity subsolution, and hence   $u$ is a viscosity supersolution.   
This completes the proof.
\end{proof}

\section{Numerical experiments and applications to learning}
\label{sec:num}

Here, we present numerical experiments and applications to learning theory. 

To solve the graph $\infty$-Laplace equation \eqref{eq:optinf}, we iterate the gradient descent-type scheme
\begin{equation}\label{eq:iteration}
u^{k+1}(x) = u^k(x) + \frac{1}{2M}Lu^k(x),
\end{equation}
where
\[Lu(x):=\max_{x\in \X}w(x,y)(u(y) - u(x)) +  \min_{x\in \X}w(x,y)(u(y) - u(x))\]
is the graph $\infty$-Laplacian, and $M:=\max_{(x,y)}w(x,y)$. The time step $\delta t = 1/(2M)$ is the largest possible while ensuring stability, and is obtained by ensuring the scheme is monotone (increasing in $u^k(x)$ on the right hand side). This is a standard stability (CFL) condition in numerical PDEs, and ensures the maximum principle holds for the scheme.  We set the Dirichlet condition $u(x)=g(x)$ for $x\in \O$ at each step. In the code, we normalize the weight matrix so that $M=1$ and run the iteration \eqref{eq:iteration} until $\max_x|L u^k(x)|< 10^{-5}$. 

We also experimented with the semi-implicit method presented in \cite{flores2019algorithms}, and found the simple iteration \eqref{eq:iteration} was faster for $\alpha>0$. We tried the code from \cite{kyng2015algorithms} and found it ran out of memory on examples with $n=10^4$ vertices that were not extremely sparse on a laptop with $16$GB of RAM, even when calling their Java code from the command line. For all algorithms we tried, the complexity of solving \eqref{eq:optinf} seems to increase with increasing $\alpha$. In the iteration \eqref{eq:iteration} $\alpha$ affects the construction of the weight matrix $w(x,y)$ as per \eqref{eq:weights}. For example, on MNIST, it takes roughly 2.5 seconds to solve \eqref{eq:optinf} with $\alpha=0$ and $8.7$ seconds with $\alpha=0.5$. For MNIST $\alpha=0.5$ is the largest value we used (the results deteriorated for larger $\alpha$).  Our code is implemented in C and is available on the author's website. 

\subsection{Visualizing the learned surface}
\label{sec:surf}

\begin{figure}
\centering
\subfigure[Uniform samples]{\includegraphics[clip=true,trim = 30 20 30 20, width=0.32\textwidth]{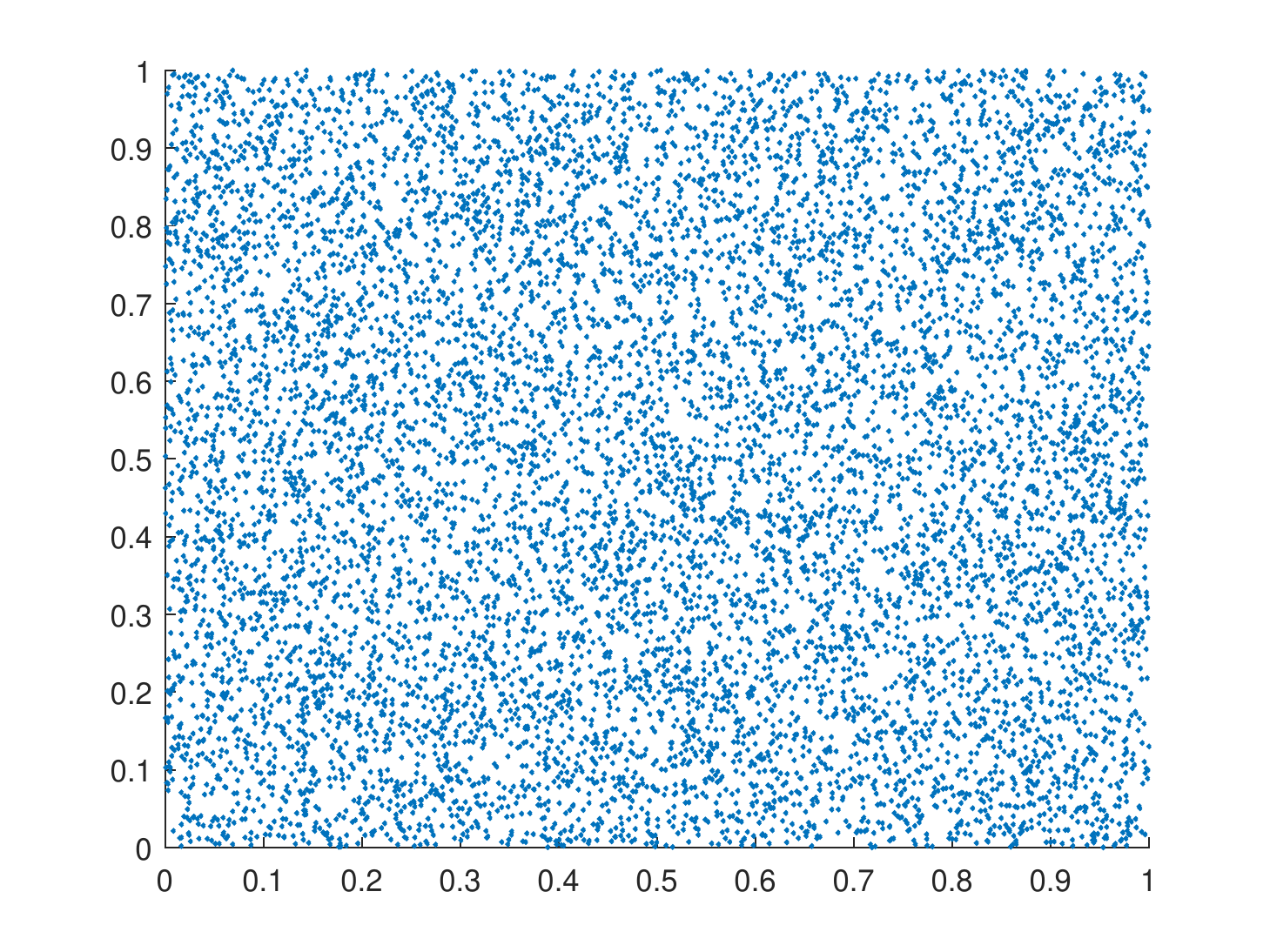}}
\subfigure[Learned function ($\alpha=0$)]{\includegraphics[clip=true,trim = 30 20 30 20, width=0.32\textwidth]{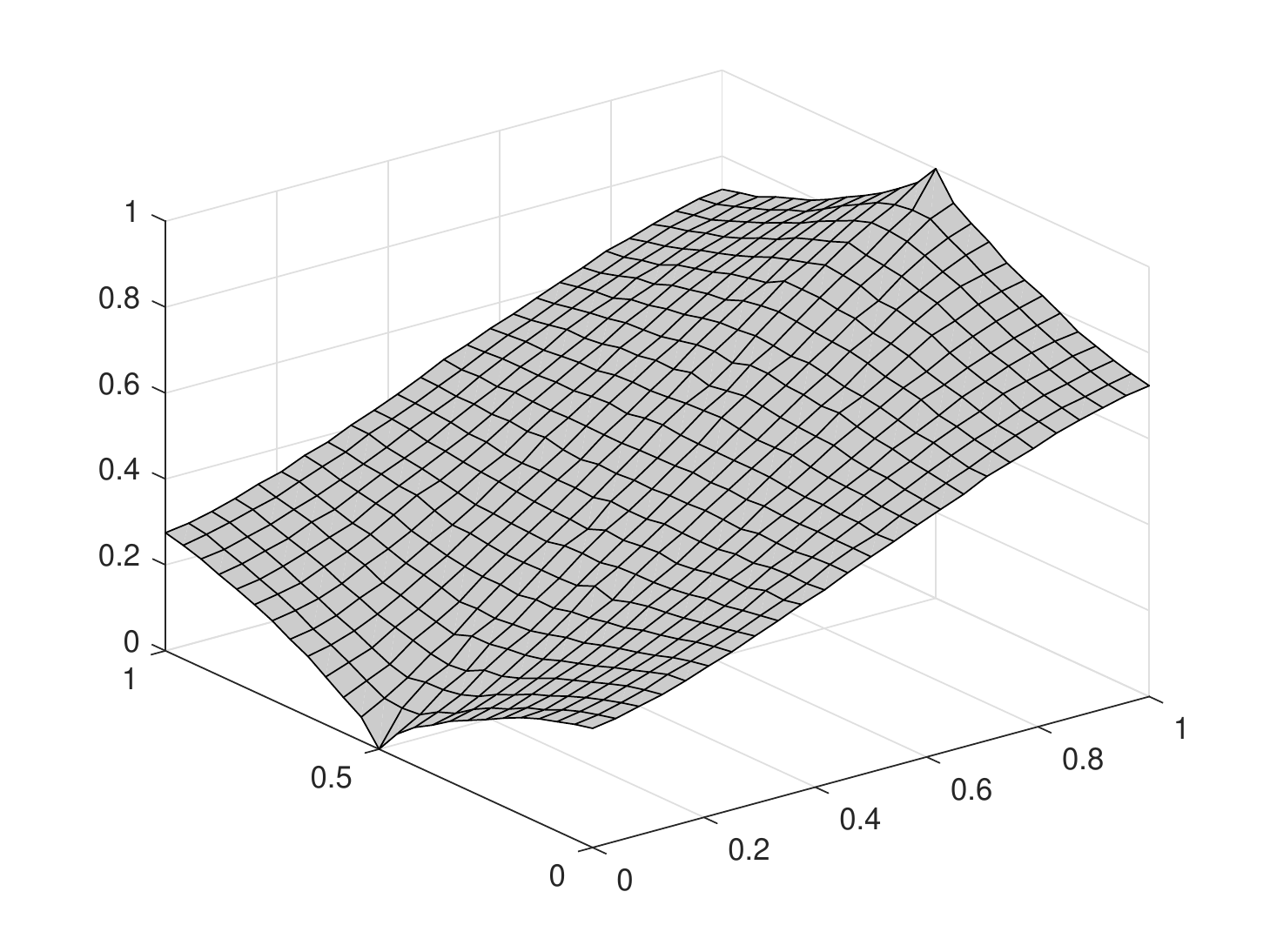}}
\subfigure[Learned function ($\alpha=1$)]{\includegraphics[clip=true,trim = 30 20 30 20, width=0.32\textwidth]{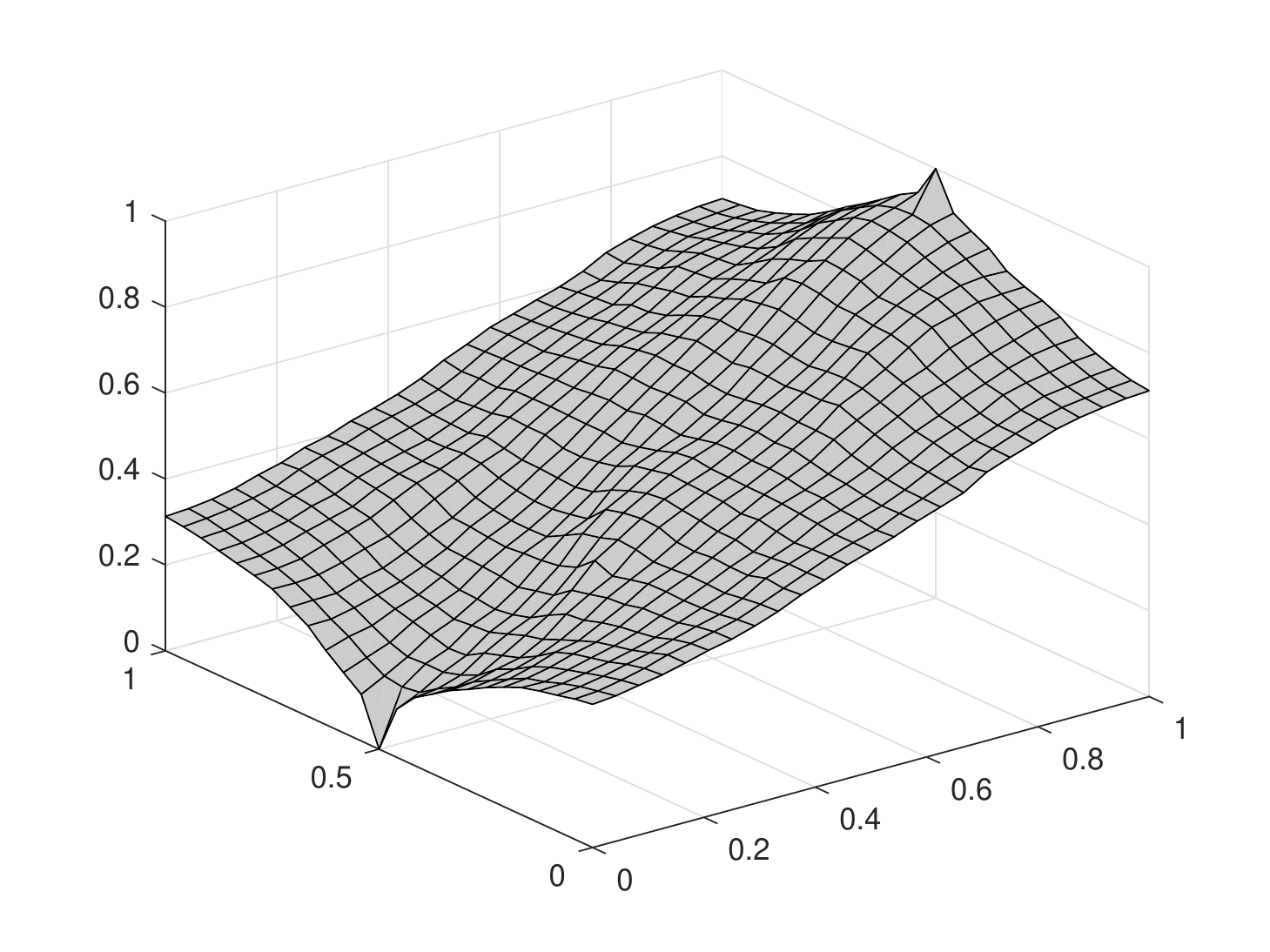}\label{fig:uniformc}}
\caption{Example of Lipschitz learning with self-tuning weights on a graph with uniformly sampled vertices.}
\label{fig:uniform}
\end{figure}

Our first simulation is designed to visualize the affect of the self-tuning weights on the learned function $u_n$. We consider graphs generated by $n=10,000$ random points on the unit box $[0,1]^2$. We choose $h = 0.05$ and the weights are selected to be 
\begin{equation}\label{eq:wei}
\sigma(x,y) = 
\begin{cases}
1,&\text{if }|x-y|\leq h\\
0,&\text{otherwise,}
\end{cases}
\end{equation}
and then $w(x,y)$ is defined as in \eqref{eq:weights}. We assign two labeled points $g(0,0.5)=0$ and $g(1,0.5)=1$, so we have $9,998$ unlabeled points, and $2$ labeled points. In Figure \ref{fig:uniform}, we show the result of Lipschitz learning for the $n=10,000$ \emph{i.i.d.}~random variables uniformly distributed on the box. In this case, the solution of Lipschitz learning does not depend on the parameter $\alpha$, since the distribution of the unlabeled points is constant. We remark the solution in Figure \ref{fig:uniformc} with $\alpha=1$ appears slightly rougher; this is due to the fact that the kernel density estimations $d_n(x)$ contain some random fluctuations, which pass to the weights in the graph. While the fluctuations in the density estimation pass through to the learned function $u_n(x)$, the learned function is still Lipschitz continuous, so it possesses  the same regularity as for $\alpha=0$. To make the surface appear smoother, one could use a larger bandwidth in the kernel density estimator, favoring lower variance and larger bias.

\begin{figure}
\centering
\subfigure[Non-uniform samples]{\includegraphics[clip=true,trim = 30 20 30 20, width=0.32\textwidth]{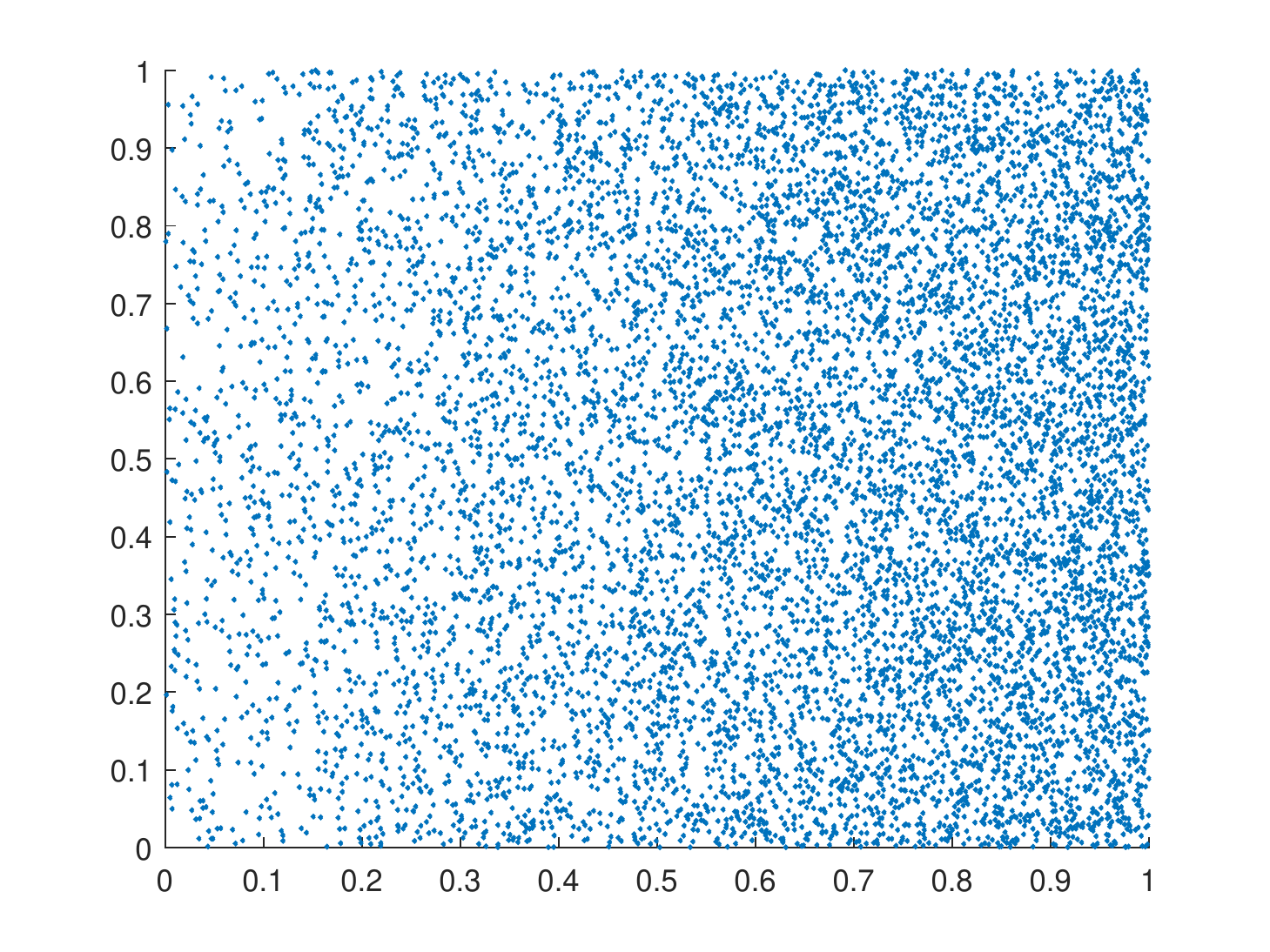}}
\subfigure[Learned function ($\alpha=0$)]{\includegraphics[clip=true,trim = 30 20 30 20, width=0.32\textwidth]{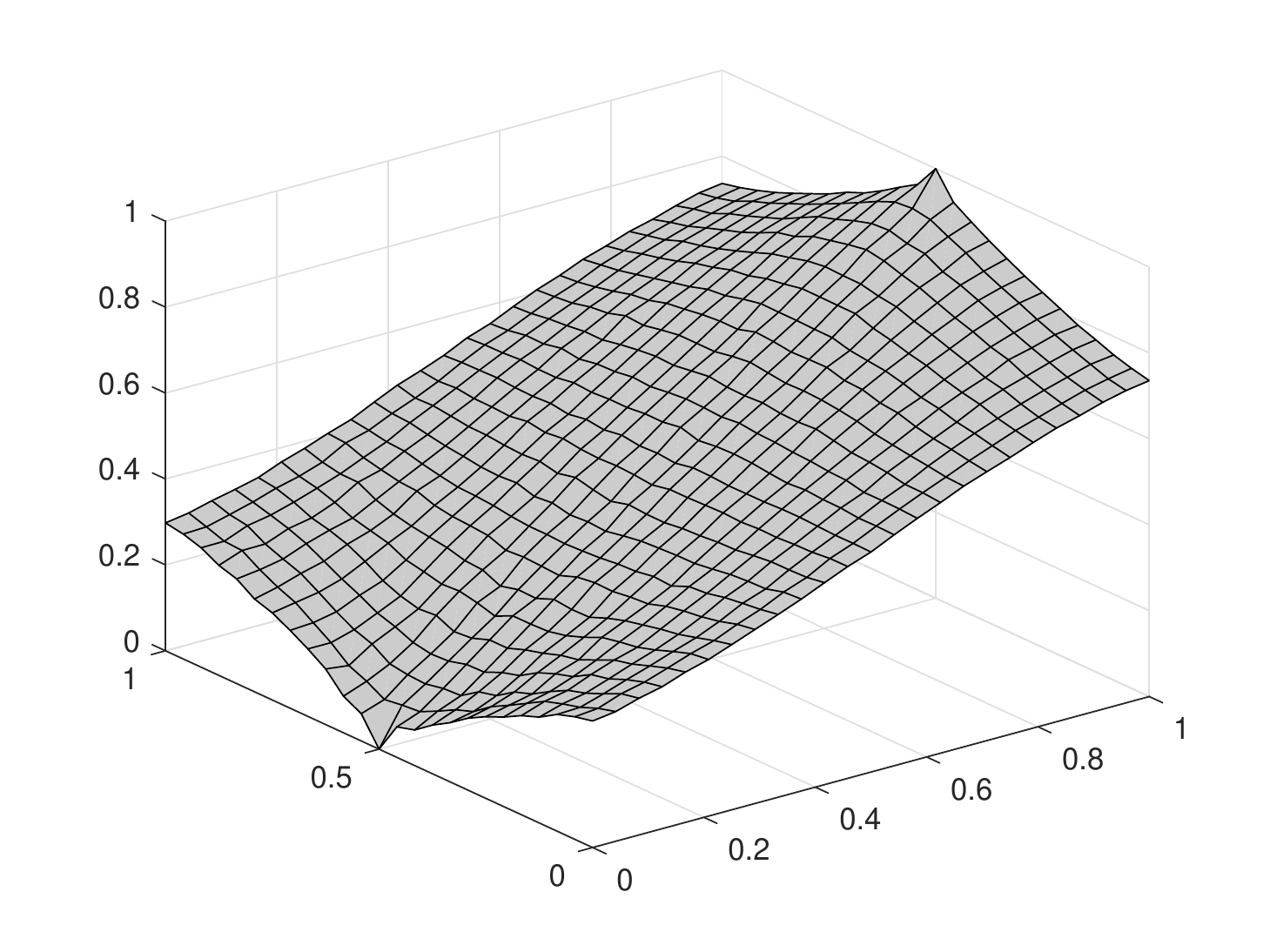}\label{fig:nonuniformb}}
\subfigure[Learned function ($\alpha=1$)]{\includegraphics[clip=true,trim = 30 20 30 20, width=0.32\textwidth]{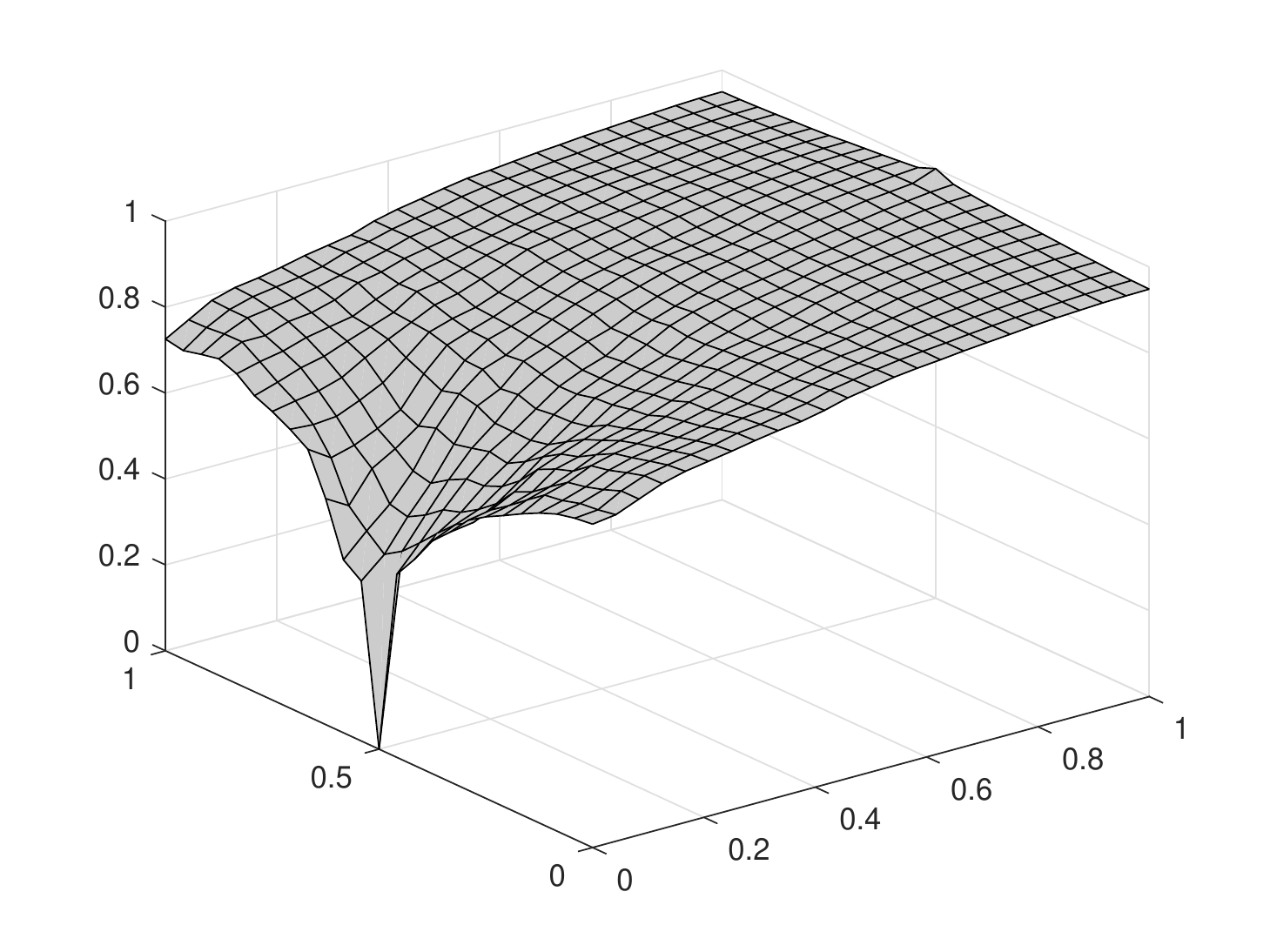}\label{fig:nonuniformd}}
\caption{Example of Lipschitz learning with self-tuning weights on a graph with non-uniformly sampled vertices. }
\label{fig:nonuniform}
\end{figure}

In Figure \ref{fig:nonuniform}, we present the same example with $n=10,000$ \emph{i.i.d.}~random variables drawn from the probability density $f(x) =\tfrac{1}{4} + \tfrac{3}{2}x_1$ on the box $[0,1]^2$. We see in Figure \ref{fig:nonuniformb} that Lipschitz learning without self-tuning weights (i.e., $\alpha=0$) gives roughly the same result as for uniformly distributed data; that is, the algorithm is insensitive to the distribution of the unlabeled data. However, we see in Figure \ref{fig:nonuniformd} that as $\alpha$ is increased, Lipschitz learning with self-tuning weights begins to feel the distribution of the unlabeled data, and places more trust in the label in the denser region. 

\subsection{An analytic example}
\label{sec:analytic}

We now study a one dimensional problem analytically, and show how self-tuning weights improve generalization performance. Due to the high degree of nonlinearity in the $\infty$-Laplace equation \eqref{eq:infLap2}, it is generally impossible to obtain closed form solutions in interesting cases for dimension larger than one. In Section \ref{sec:synth}, we give a numerical study of the higher dimensional version of this classification problem.

We assume our data lies in the interval $\Omega=[-1,1]$.\footnote{It is straightforward to extend the setup to periodic boundary conditions, as in Theorems \ref{thm:pinf} and \ref{thm:pinf2}, but the presentation is more cumbersome.} We assume our unlabeled data has distribution $f$ given by
\begin{equation}\label{eq:fdef}
f(x)=
\begin{cases}
A,&\text{if }\delta \leq |x|\leq 1\\
\mu A,&\text{if } |x|\leq \delta,
\end{cases}
\end{equation}
where $\mu,\delta \in (0,1)$ are parameters, and $A>0$ is chosen so that $\int_{-1}^1 f(x)\, dx =1$, that is,
\[A = \frac{1}{2(\delta\mu + 1-\delta)}.\]
The distribution $f$ has a dip in the region $[-\delta,\delta]$ of relative magnitude $\mu$, indicating the transition region between two labels. A similar example was also considered recently in \cite{calder2018properly}. In particular, we assume the true label function $g:\Omega\to \R$ is
\begin{equation}\label{eq:gdef}
g(x)=
\begin{cases}
1,&\text{if } 0 \leq x\leq 1\\
-1,&\text{if } -1 \leq x < 0.
\end{cases}
\end{equation}
For fixed $x_1,x_2\in [\delta,1]$, we assume we are given exactly two labels 
\begin{equation}\label{eq:given}
g(-x_1) = -1 \ \ \text{ and }\ \ g(x_2) =1.
\end{equation}
The unlabeled data points are sampled independently from the distribution $f$. We construct a graph with self-tuning weights \eqref{eq:weights} with parameter $\alpha$, and solve the Lipschitz learning problem \eqref{eq:optinf} to obtain a classifier on the unlabeled data. We can then compute the classification accuracy as the fraction of unlabeled data points that are labeled correctly according to \eqref{eq:gdef}.

To analyze classification accuracy, and how it depends on $\alpha$ and $\mu$,  we solve the continuum problem \eqref{eq:infLap2} instead of solving the graph-based problem \eqref{eq:optinf}. Theorem \ref{thm:pinf2} guarantees the approximation error will be small for sufficiently many vertices in the graph, so the approximation is justified. For $\alpha\in \R$, let $u_{\alpha}$ denote the solution of the continuum $\infty$-Laplace equation \eqref{eq:infLap2}, which represents the continuum limit of the Lipschitz learning with self-tuning weights. Hence, $u_\alpha$ solves the Euler-Lagrange equation
\begin{equation}\label{eq:uaode}
u_\alpha'' + 2\alpha \frac{f'(x)}{f(x)}u_\alpha'(x) = 0 \ \ \text{ for }x \in (-1,1)\setminus \{-x_1,x_2\},
\end{equation}
subject to $u_\alpha(-x_1)=-1$, $u(x_2)=1$ and the Neumann condition $u_\alpha'(-1)=u_\alpha'(1)=0$. Equivalently, $u_\alpha$ can be characterized as the minimizer (in the absolutely minimal sense) of 
\begin{equation}\label{eq:variational}
E[u]:=\max_{x\in [-1,1]} \{f(x)^{2\alpha}|u'(x)|\}
\end{equation}
among all functions $u:[-1,1]\to \R$ satisfying $u(-x_1)=-1$ and $u(x_2)=1$. 

To solve for $u_\alpha$, we note that away from $x=\pm \delta$, $f'(x)=0$ and so $u_\alpha''(x)=0$ for $x\neq -x_1,x_2$. Hence $u_\alpha$ is linear on the intervals $[-1,-x_1)$, $(-x_1,-\delta)$, $(-\delta,\delta)$, $(\delta,x_2)$ and $(\delta,1]$, and continuous on $[-1,1]$. In the intervals $[-1,-x_1)$ and $(x_2,1]$ the solution must be constant, taking values $-1$ and $+1$ respectively, otherwise the Lipschitz constant can be locally improved by truncating $u_\alpha$ to be constant in these regions. In the other intervals, the quantity $f(x)^{2\alpha}|u'(x)|$ should be equal across the remaining 3 intervals, otherwise we could decrease the energy $E(u)$ by adjusting the slopes in each interval. This yields the solution formula
\begin{equation}\label{eq:ualpha}
u_\alpha(x) = 
\begin{cases}
-1,&\text{if }-1 \leq x \leq -x_1\\
m(x+x_1)-1,&\text{if }-x_1\leq x \leq -\delta\\
\mu^{-2\alpha}m(x+\delta + \mu^{2\alpha}(x_1-\delta))-1,&\text{if }-\delta \leq x \leq \delta\\
m(x-x_2) + 1,&\text{if }\delta \leq x \leq x_2\\
1,&\text{if }x_2 \leq x \leq 1,
\end{cases}
\end{equation}
where
\begin{equation}\label{eq:mdef}
m = \frac{2}{x_1 + x_2 - 2\delta + 2\delta \mu^{-2\alpha}}.
\end{equation}
We note that as $\alpha$ increases, the slope $m$ in the regions $(-x_1,-\delta)$ and $(\delta,x_2)$ decreases, and the slope in the region $(-\delta,\delta)$ increases, allowing a sharper transition between classes (note $0 < \mu < 1$). Figure \ref{fig:ualpha} shows plots of $u_\alpha$ for various values of $\mu$ and $\alpha$.
\begin{figure}
\centering
\subfigure[Varying $\alpha$]{\includegraphics[trim = 35 20 40 20, clip = true, width=0.48\textwidth]{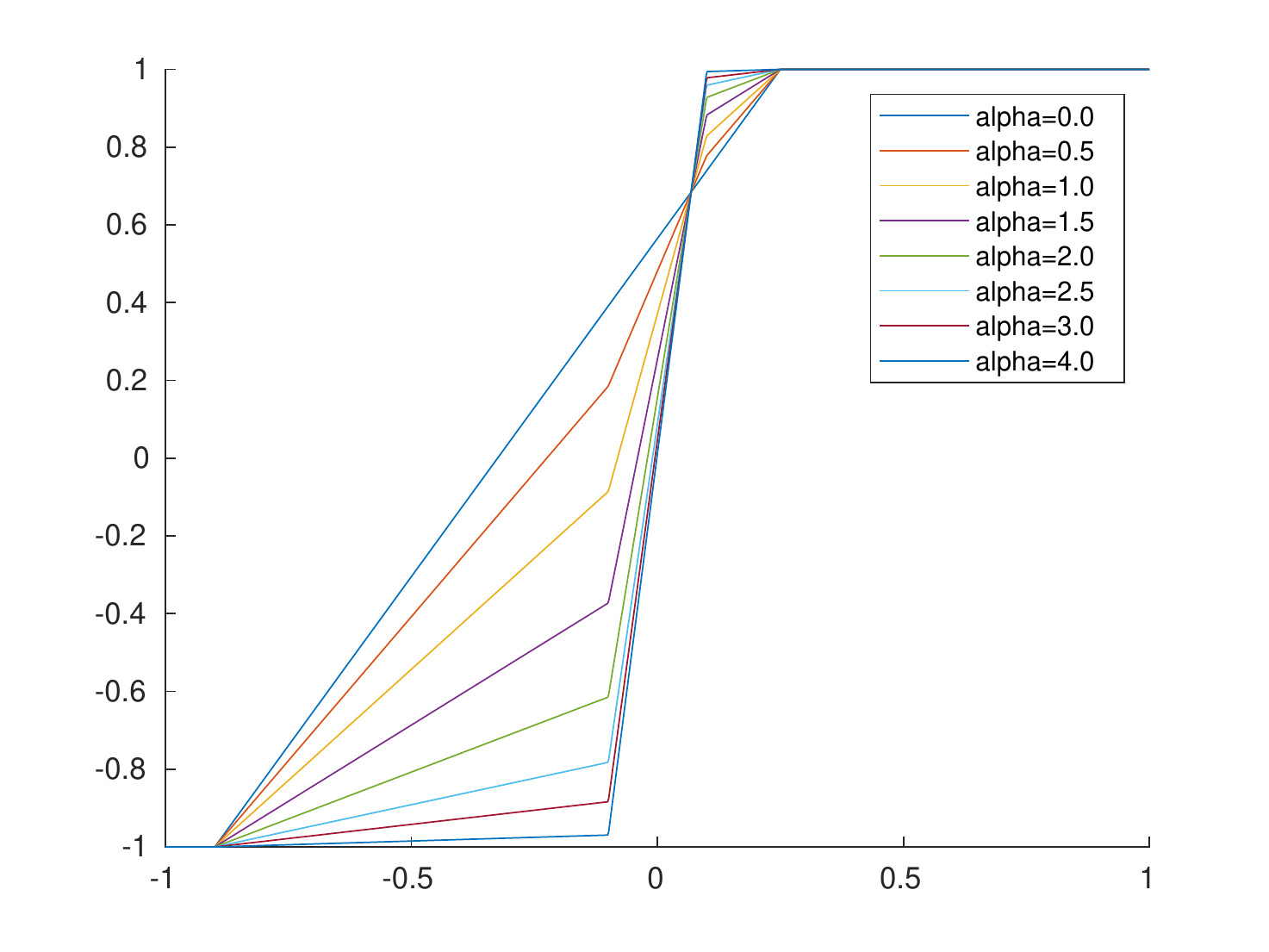}}
\subfigure[Varying $\mu$]{\includegraphics[trim = 35 20 40 20, clip = true, width=0.48\textwidth]{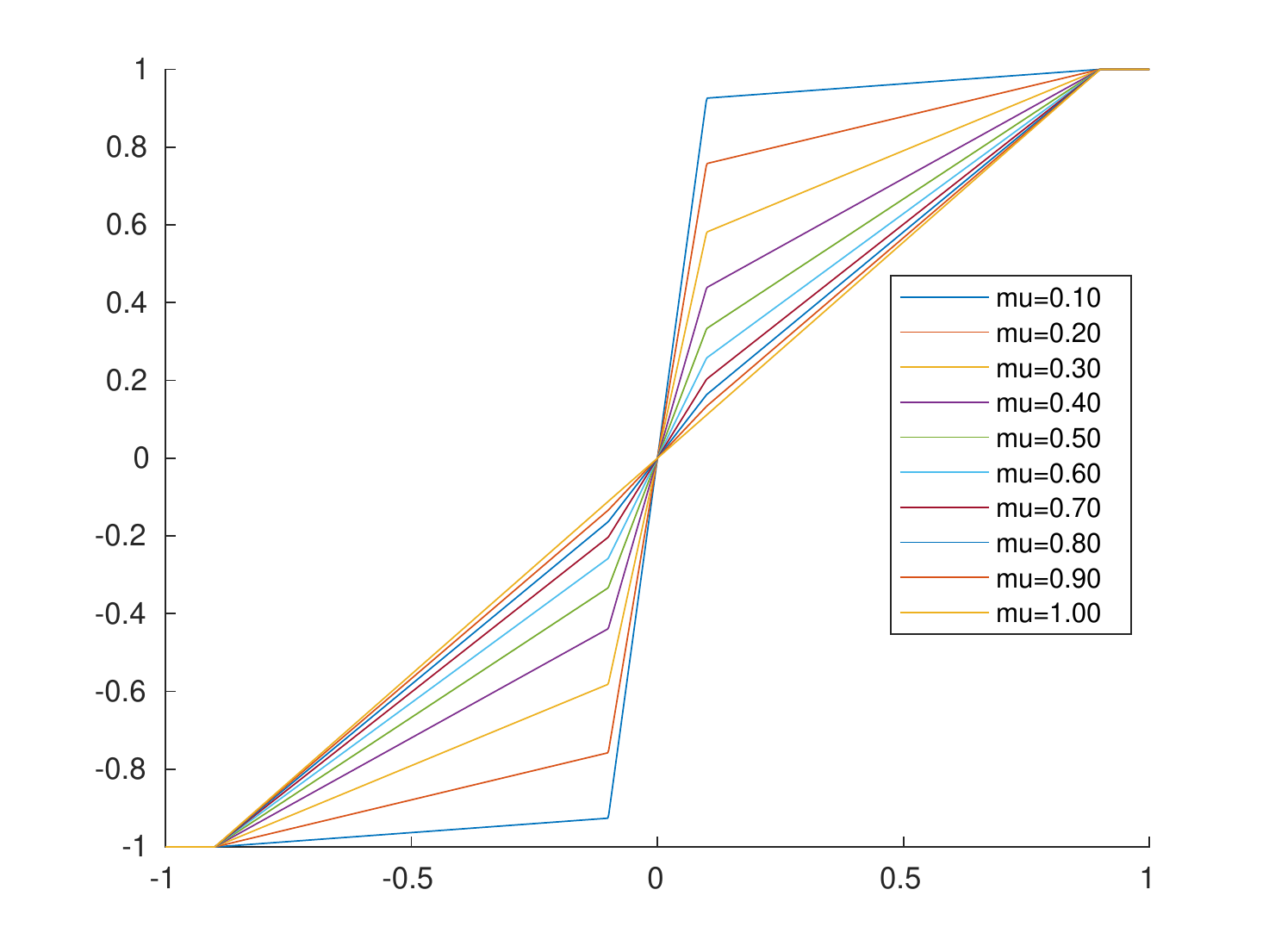}}
\caption{Plots of $u_\alpha$ for (a) varying $\alpha$ and (b) varying $\mu$. In (a) we set $x_1=0.9$, $x_2=0.25$, $\delta=0.1$, and $\mu=0.5$ and in (b) we set $x_1=0.9$, $x_2=0.9$, $\delta=0.1$, and $\alpha=1$. We observe in (a) that as $\alpha$ increases, the learned function $u_\alpha$ becomes more sensitive to the dip in the distribution near $x=0$, and prefers a sharper transition in the less dense region. In (b) we see a similar phenomenon as we change the strength  $\mu$ of the dip in the distribution. }
\label{fig:ualpha}
\end{figure}

We now analyze the classification accuracy of $u_\alpha$. For classification, the points $x$ for which $u_\alpha(x)<0$ are labeled $-1$, and the points $x$ for which $u_\alpha(x) >0$ are labeled $+1$. Thus, the classification accuracy, as a score between $0$ and $1$, is given by
\begin{equation}\label{eq:accuracy}
\text{Accuracy}=\frac{1}{2}\int_{-1}^0 (-u_\alpha(x))_+ \, dx + \frac{1}{2}\int_{0}^1 u_\alpha(x)_+ \, dx,
\end{equation}
where $t_+ := \max\{t,0\}$.  In this case, we can explicitly compute the accuracy.
\begin{proposition}\label{prop:acc}
The accuracy \eqref{eq:accuracy} can be expressed as
\begin{equation}\label{eq:acc}
\text{Accuracy}=
\begin{cases}
1 - \frac{1}{2}\delta - \frac{1}{4}\left( |x_2-x_1| - 2\delta\mu^{-2\alpha} \right),&\text{if }2\delta\mu^{-2\alpha}\leq |x_2-x_1|\\
1 - \frac{1}{4}\mu^{2\alpha}|x_2-x_1|,&\text{if }2\delta\mu^{-2\alpha}\geq |x_2-x_1|.
\end{cases}
\end{equation}
\end{proposition}
\begin{proof}
Let $a=\min\{x_1,x_2\}$ and $b=\max\{x_1,x_2\}$. Without loss of generality, let us assume $b=x_2$. Then
\[\text{Accuracy}=\frac{1}{2} + \frac{1}{2}(1-x^*),\]
where $x^*$ satisfies $u_\alpha(x^*) =0$. If $x^* > \delta$ then $x^*$ satisfies  $m(x^*-b) + 1 =0$, and so
\[x^* = \frac{1}{2}(b-a) + \delta(1-\mu^{-2\alpha}).\]
Since we assumed $x^*>\delta$ we have $2\delta\mu^{-2\alpha}\leq b-a$. In this case, the classification accuracy is
\[\text{Accuracy}=1 - \frac{1}{2}\delta - \frac{1}{4}\left( b-a - 2\delta\mu^{-2\alpha} \right).\]
If $2\delta\mu^{-2\alpha}\geq b-a$ then $x^*\in [-\delta,\delta]$ and we find that
\[x^* = \frac{1}{2}\mu^{2\alpha}(b-a).\]
Thus, in this case
\[\text{Accuracy}=1 - \frac{1}{4}\mu^{2\alpha}(b-a).\]
The proof is completed by noting $b-a=|x_2-x_1|$.
\end{proof}
Proposition \ref{prop:acc} shows that accuracy increases as $\alpha$ increases, and as $\mu$ decreases. We can interpret this in the following way: Increasing $\alpha$ makes the algorithm more sensitive to the distribution of data, and gives a higher preference to placing a decision boundary in the interval $[-\delta,\delta]$ where the distribution dips, while decreasing $\mu$ results in a larger dip in the data distribution, which is easier to detect by the algorithm. 

We also note that since $0 < \mu < 1$, the accuracy converges to $1$ (perfect classification) as $\alpha\to \infty$. On the other hand, if $\alpha=0$ then
\[\text{Accuracy}= 1 - \frac{1}{4}|x_2-x_1|,\]
that is, the algorithm becomes insensitive to the distribution of the unlabeled data.  If $x_2=x_1$, then we always have $\text{Accuracy}=1$, simply due to symmetry in the problem, forcing the zero crossing to $x^*=0$.

\subsection{A synthetic classification example}
\label{sec:synth}

\begin{figure}
\centering
\subfigure[Data]{\includegraphics[trim = 35 20 40 20, clip = true, width=0.48\textwidth]{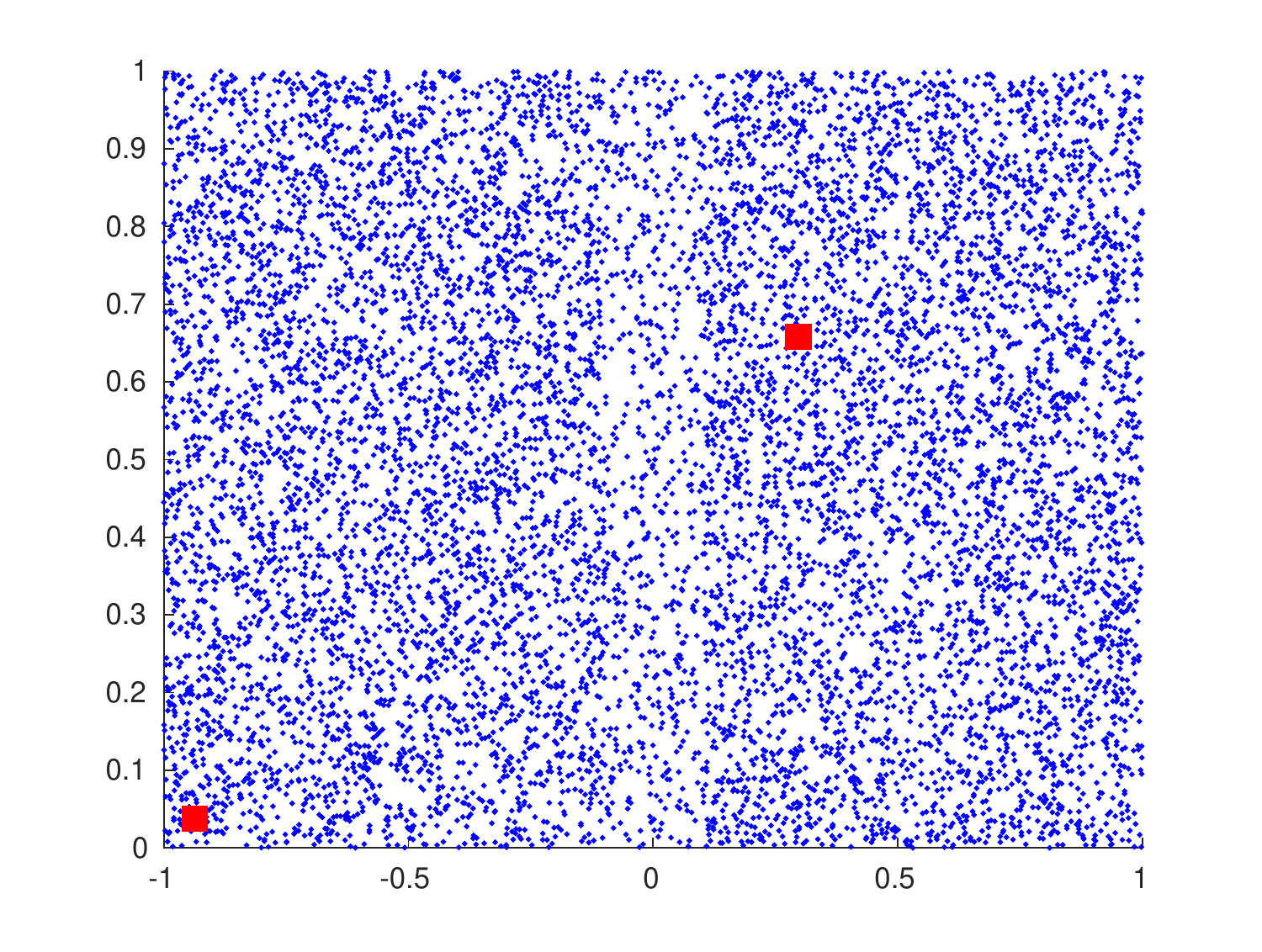}}
\subfigure[$\alpha=0$ $ (81.5\%)$]{\includegraphics[trim = 35 20 40 20, clip = true, width=0.48\textwidth]{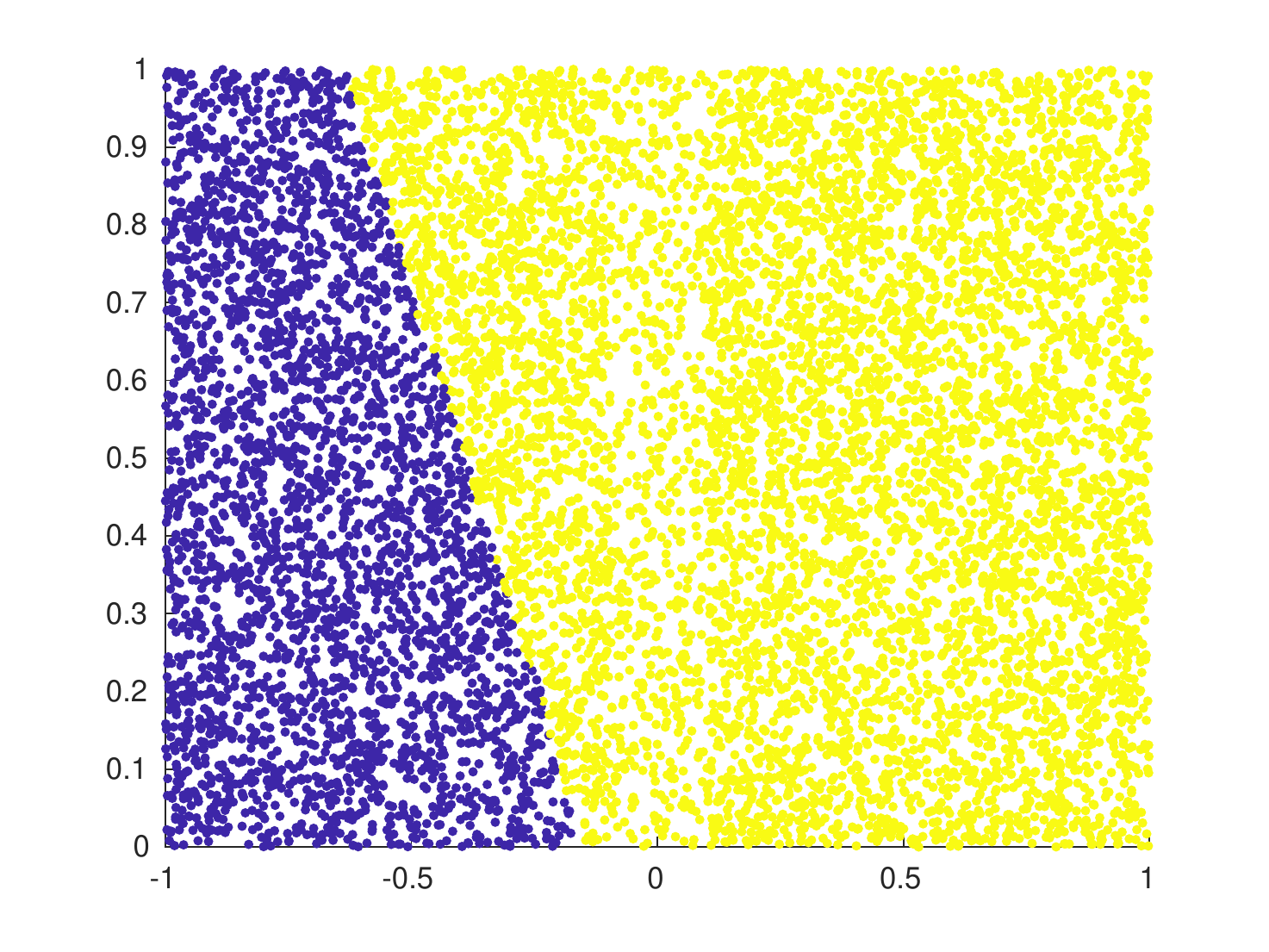}}
\subfigure[$\alpha=1$ $ (93.44\%)$]{\includegraphics[trim = 35 20 40 20, clip = true, width=0.48\textwidth]{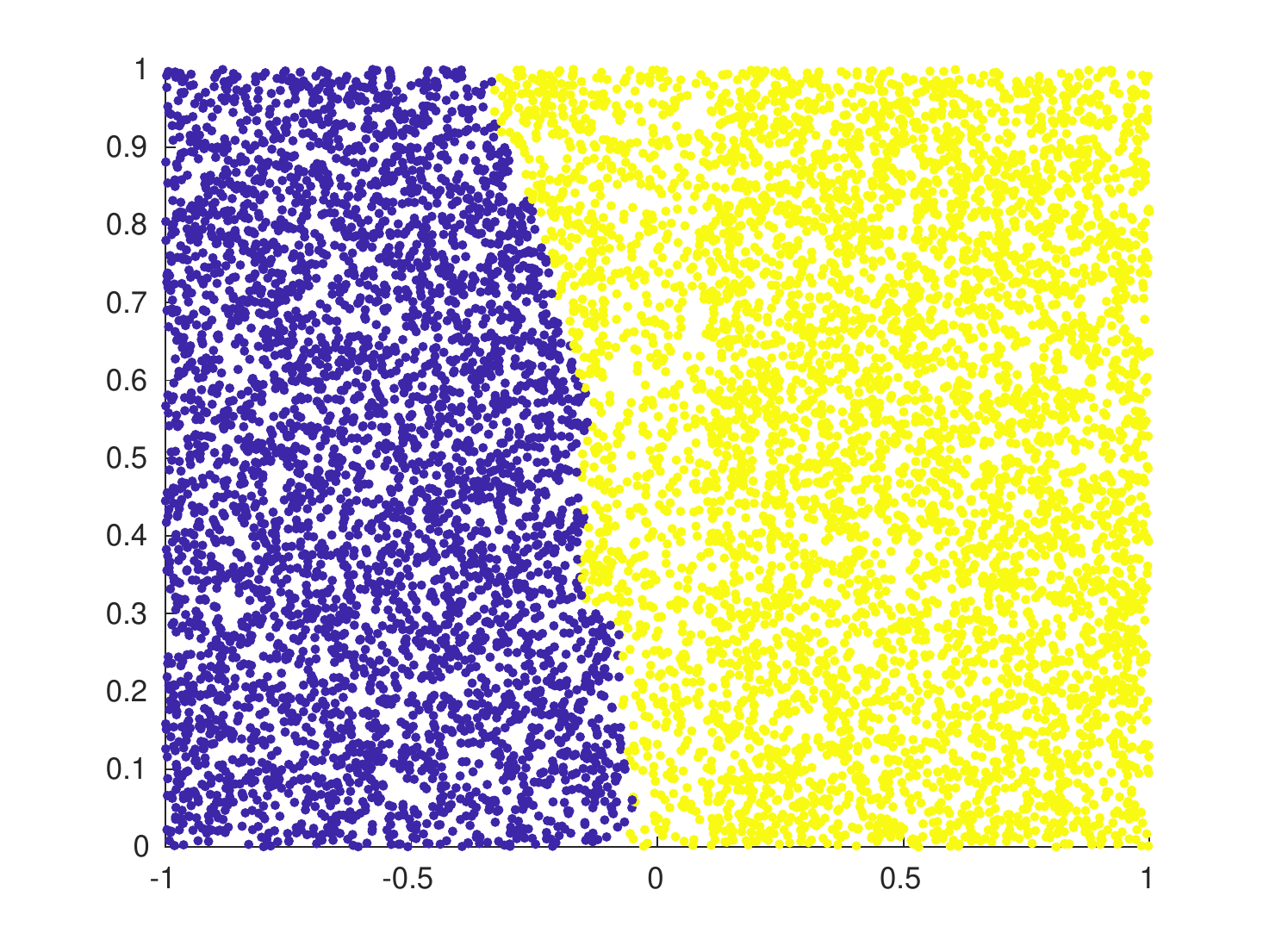}}
\subfigure[$\alpha=2$ $ (98.88\%)$]{\includegraphics[trim = 35 20 40 20, clip = true, width=0.48\textwidth]{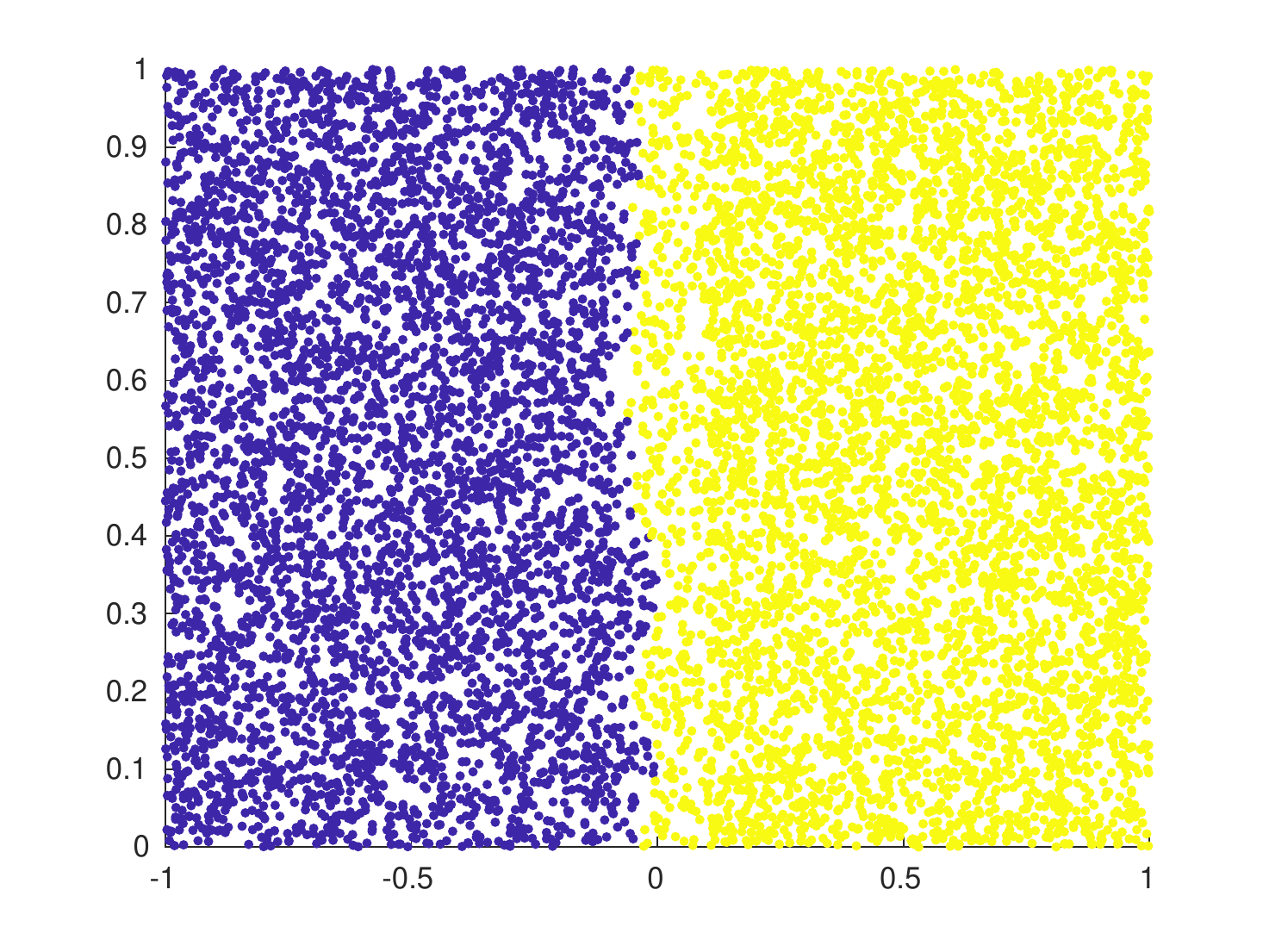}}
\caption{Examples of classification for different $\alpha$. The labels are given at the red points in (a) in this realization. }
\label{fig:syn_examples}
\end{figure}

\begin{figure}
\centering
\subfigure[$\alpha\in (0,2),\mu=0.5$]{\includegraphics[trim = 0 0 30 20, clip = true, width=0.48\textwidth]{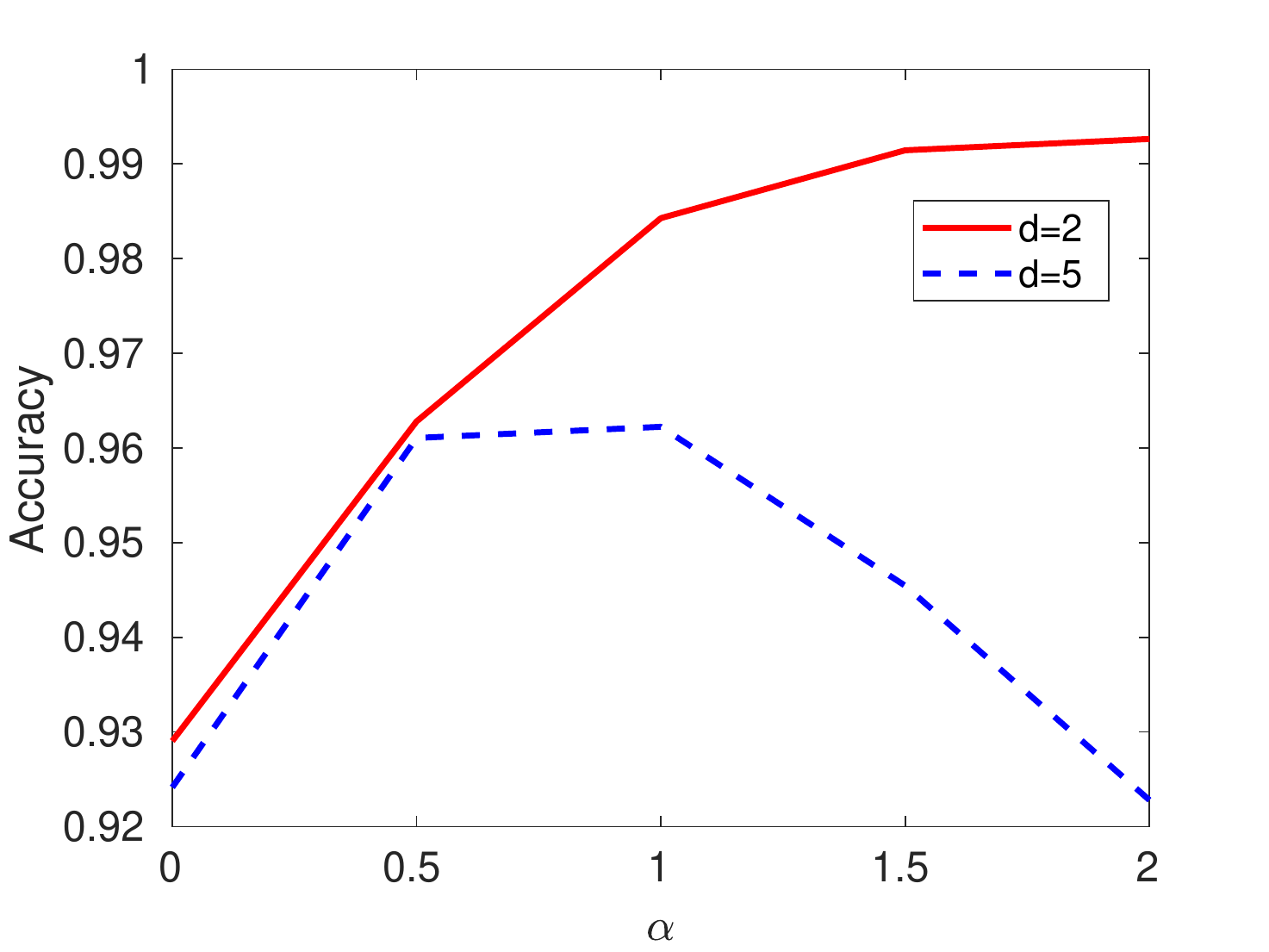}}
\subfigure[$\alpha=1,\mu\in (0.5,0.9)$]{\includegraphics[trim = 0 0 30 20, clip = true, width=0.48\textwidth]{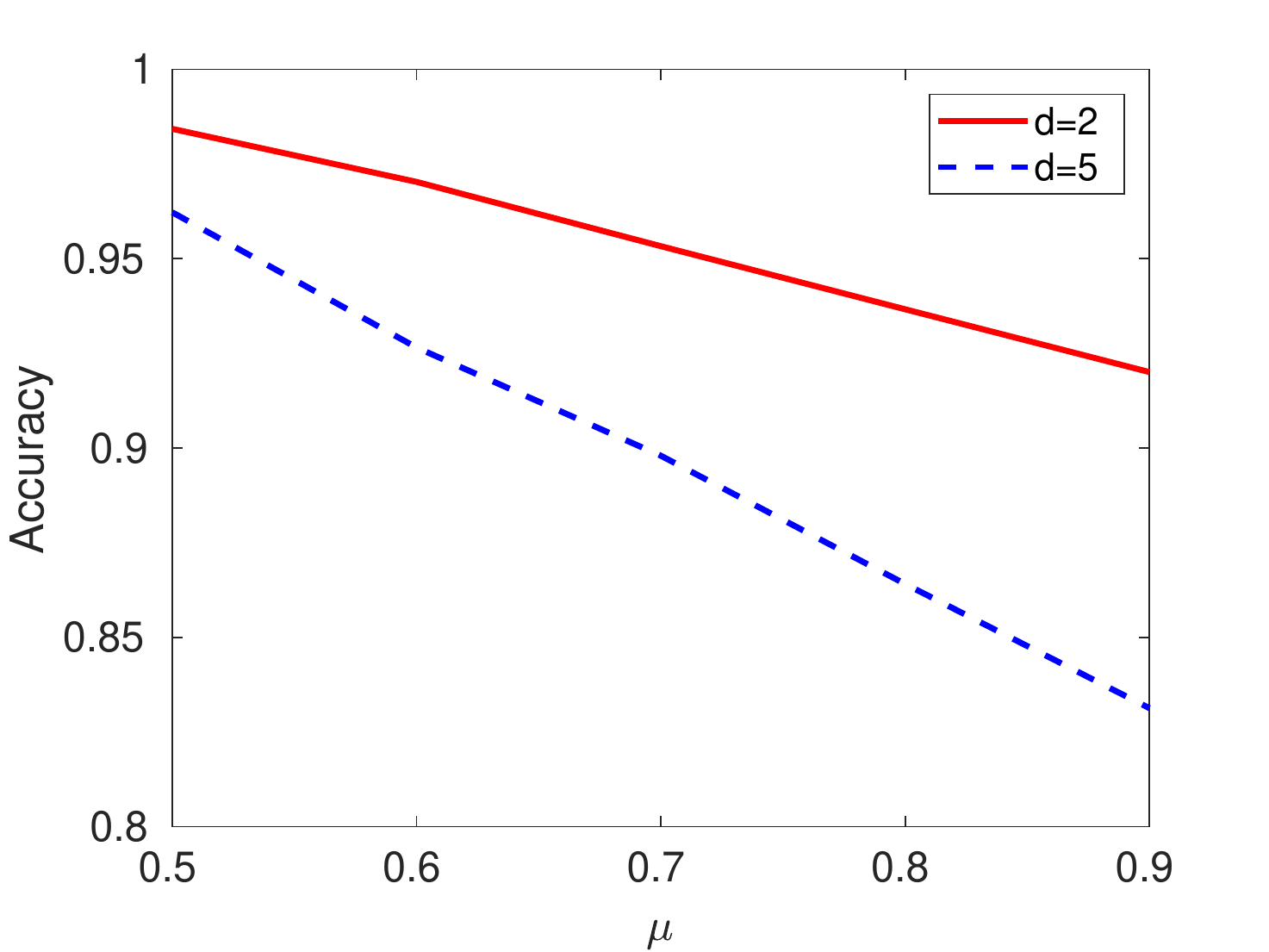}}
\caption{Results of simulations demonstrating how $\alpha$ and $\mu$ affect the accuracy in dimensions $d=2$ and $d=5$. Each experiment is averaged over 100 trials.}
\label{fig:alphamu}
\end{figure}

Here we examine the analytic example from Section \ref{sec:analytic} in higher dimensions. In this case, we cannot solve the PDE \eqref{eq:infLap2} in closed form, so instead we present the results of numerical simulations. 

Our domain is $\Omega=[-1,1]\times [0,1]^{d-1}$. The unlabeled data follows the distribution 
\begin{equation}\label{eq:ulab}
f(x) = 
\begin{cases}
A,&\text{if }\delta \leq |x_1| \leq 1\\
\mu A,&\text{if }|x_1|\leq \delta
\end{cases}
\end{equation}
where $\delta,\mu\in (0,1)$ and $A>0$ is chosen so that $f$ is a probability density. As in Section \ref{sec:analytic}, the distribution $f$ has a dip in density near $x_1=0.5$, which indicates the transition between labels. The true label function is
\begin{equation}\label{eq:gdef2}
g(x)=
\begin{cases}
1,&\text{if } 0 \leq x_1\leq 1\\
-1,&\text{if } -1 \leq x_1 < 0.
\end{cases}
\end{equation}
Our unlabeled data is a sequence $Y_1,\dots,Y_n$ of $n$ \emph{i.i.d.}~random variables with density $f$. For our labeled data we provide exactly two labels $g(-X_1)=-1$ and $g(X_2) = 1$, where $X_1,X_2$ are independent random variables uniformly distributed on $[\delta,1]\times [0,1]^d$. Given the labeled and unlabeled data described above, we generate the graph weights according to \eqref{eq:weights}, for $h,\alpha$ to be specified, and we solve the $\infty$-Laplace learning problem \eqref{eq:optinf}. The learned function is thresholded at $u=0$ to obtain the final classification.

Figure \ref{fig:syn_examples} shows the learned functions for $\alpha=0,1,2$ for a single realization of this experiment. We see that as $\alpha$ is increased, the learned function pays more attention to the distribution and places the decision boundary closer to the region where the distribution dips at $x_1=0.5$. When $\alpha=2$ we get nearly perfect classification. In this example we chose $h=0.05$, $\delta=0.1$ and $\mu=0.5$.

We ran this experiment for different values of $\alpha$ and $\mu$, each time averaging over 100 trials of the experiment. Figure \ref{fig:alphamu} shows the average classification accuracy for the $d=2$ and $d=5$ dimensional cases. For $d=2$ we used $\delta=0.1$ and $h=0.5$, and for $d=5$ we used $h=0.25$ and $\delta=0.2$ in the $d=5$ experiment. In both cases $\alpha$ and $\mu$ are varied between $0$ and $2$, and $0.5$ to $0.9$, respectively. We see in Figure \ref{fig:alphamu} that for $d=2$, accuracy is always increasing with $\alpha$. However, for $d=5$, there is an optimal $\alpha$ (near $\alpha=1$), and performance degrades for larger $\alpha$. This is the situation we expect in practice; when $\alpha$ is too large, the algorithm begins to feel the fluctuations (variance) in the kernel density estimator too much, and is trying to fit noise. We also see in Figure \ref{fig:alphamu} that accuracy decreases with increasing $\mu$, which is to be expected since the dip in the distribution is smaller when $\mu$ is larger, and thus harder to detect.

\subsection{MNIST}
\label{sec:MNIST}

\begin{figure} \begin{center} 
\includegraphics[width = 0.6\textwidth]{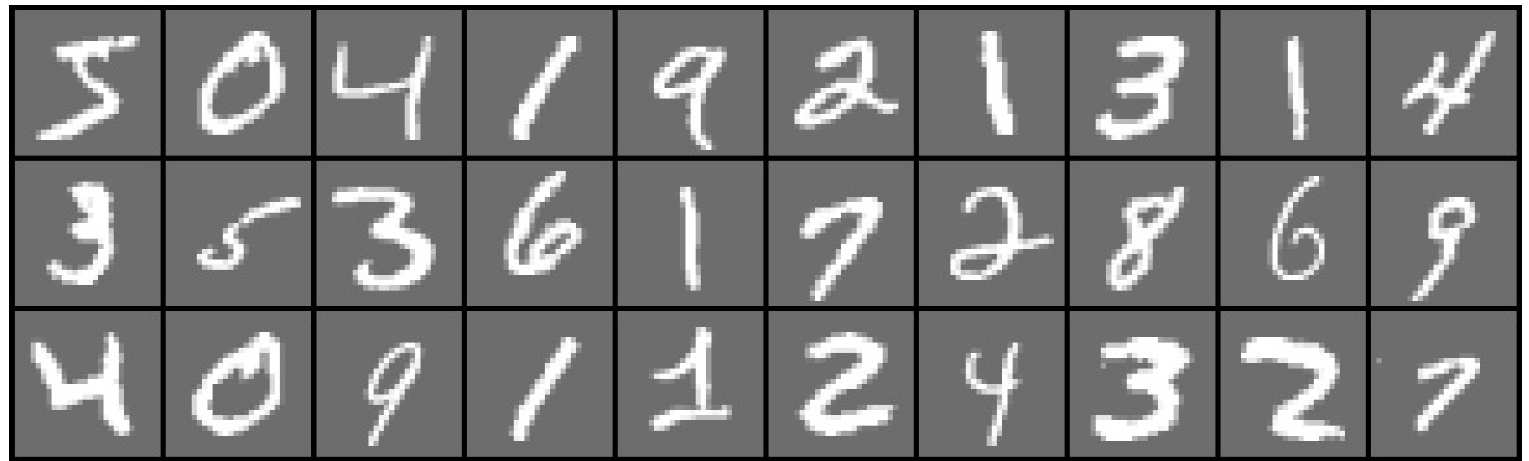}
\caption{Sample handwritten digits from the MNIST dataset. \label{fig:sample_mnist_digits}}
\end{center} \end{figure}

\begin{figure}
\centering
\subfigure[MNIST]{\includegraphics[width=0.48\textwidth,clip=true,trim=10 0 20 15 ]{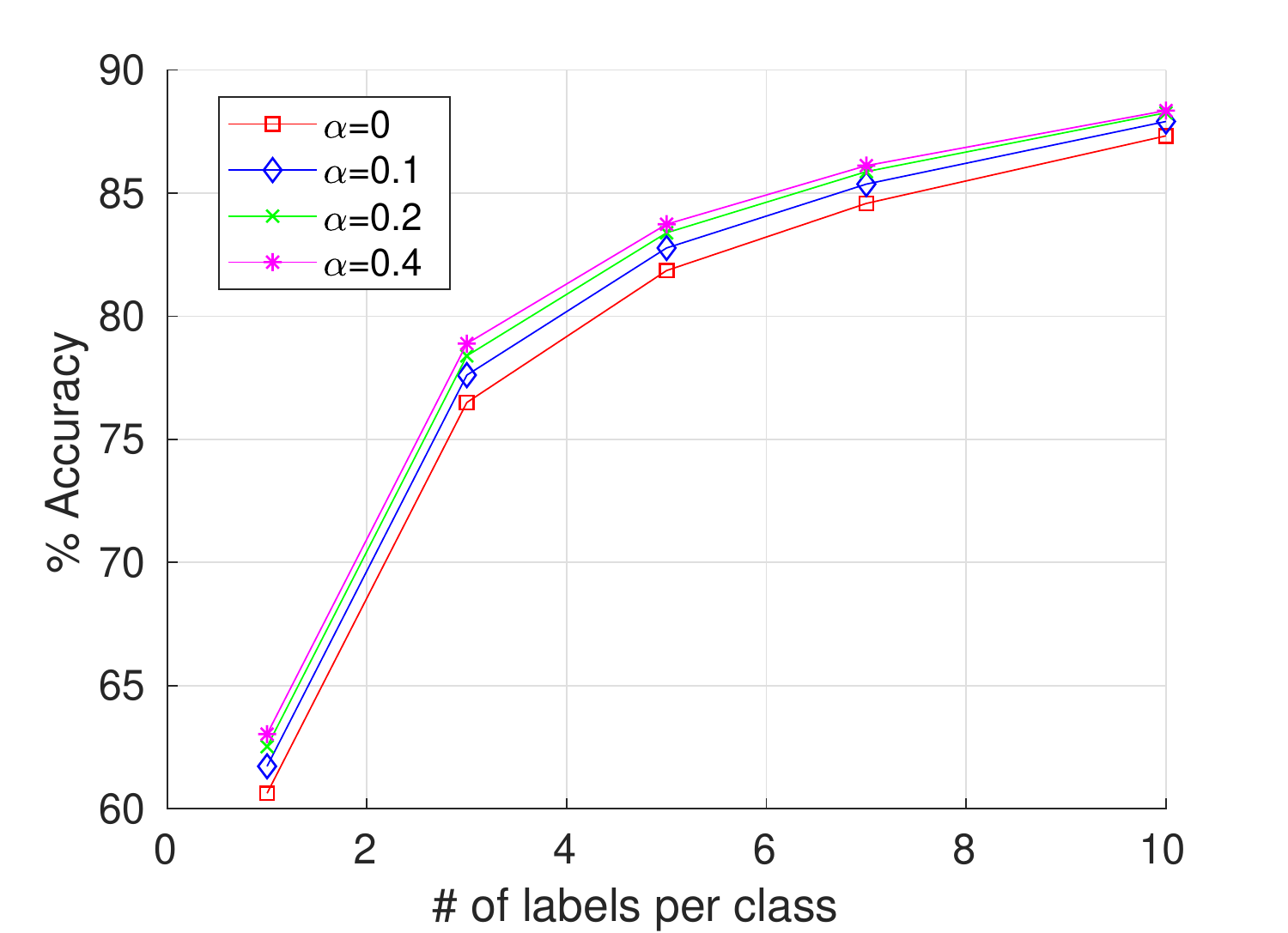}}
\subfigure[MNIST]{\includegraphics[width=0.48\textwidth,clip=true,trim=0 0 20 15 ]{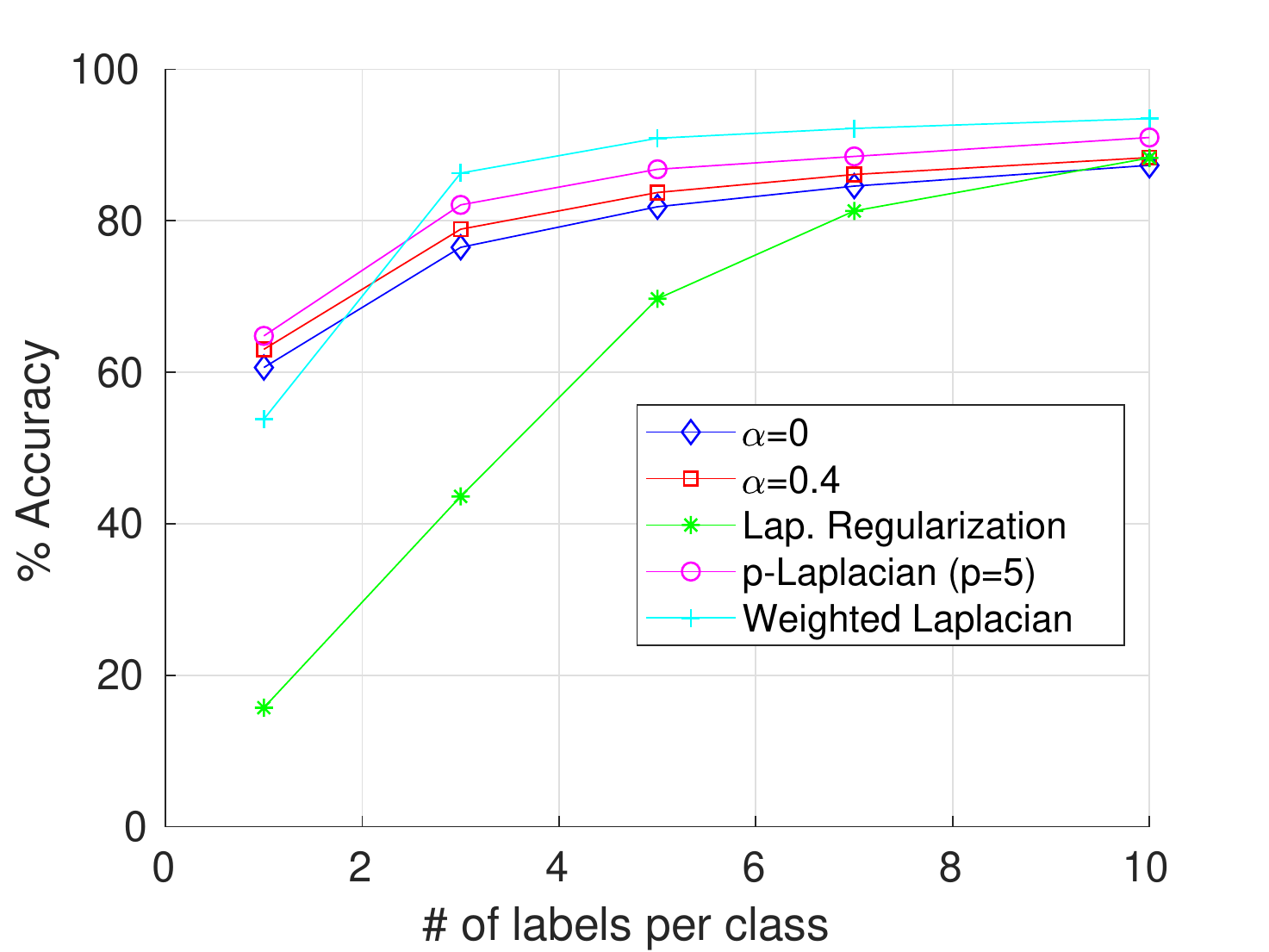}}
\caption{MNIST}
\label{fig:MNIST}
\end{figure}

We now present experiments with the MNIST dataset of handwritten digits \cite{lecun1998gradient}. The dataset consists of $70,000$ $28\times 28$ pixel grayscale images of handwritten digits $0-9$. Figure \ref{fig:sample_mnist_digits} shows an example of the MNIST digits. Our construction of the graph over MNIST is the same as in \cite{calder2018properly}. We connect each image to its nearest $10$ Euclidean neighbors, and assign Gaussian weights with $\sigma$ the distance to the $5^{\rm th}$ nearest neighbor. We then symmetrize the graph by replacing the weight matrix $W$ with $\tfrac{1}{2}(W^T + W)$. The self tuning weights are defined as in Remark \ref{rem:knngraphs} (see \eqref{eq:selfknn}).

In our experiments, we take between $1$ and $10$ labels per digit (so $10$ up to $100$ labels total) chosen at random, and average the accuracy over $100$ trials. To perform multi-class classification, we solve the $\infty$-Laplace equation \eqref{eq:optinf} $10$ times, for each digit versus the rest, giving 10 probabilities for each unlabeled image, and the label is assigned by the digit with maximal probability. The algorithm is standard in semi-supervised learning, and identical to the one used in \cite{calder2018properly,shi2017weighted}. Figure \ref{fig:MNIST}(a) shows the accuracy as a function of number of labels for $\alpha=0,0.1,0,2,0.4$. We see accuracy improves with self-tuning weights. We found no further improvement beyond $\alpha=0.4$. In Figure \ref{fig:MNIST}(b), we compare against recent graph-based algorithms for semi-supervised learning with few labels, including the weighted Laplacian \cite{shi2017weighted}, the game-theoretic $p$-Laplacian with $p=5$ \cite{calder2018game,flores2019algorithms} and classical Laplacian regularization \cite{zhu2003semi}. We see that Lipschitz learning with self-tuning weights is competitive in the regime with very few labels.

\section{Conclusions}
\label{sec:conc}

In this paper, we proved that Lipschitz learning is well-posed in the limit of infinite unlabeled data, and finite labeled data. Furthermore, contrary to current understandings of Lipschitz learning, we showed that the algorithm can be made highly sensitive to the distribution of unlabeled data by choosing self-tuning weights in the construction of the graph. Our results followed by proving the sequence of learned functions converges to the viscosity solution of an $\infty$-Laplace type equation, and then studying properties of that equation. Our results are unique in the context of consistency of graph Laplacians in that they use very minimal probability, which is a feature of the graph $\infty$-Laplacian. In particular, our results hold in both \emph{i.i.d.}~and non-\emph{i.i.d.}~settings. We also presented the results of numerical experiments showing that self-tuning weights improve Lipschitz learning by making it more sensitive to the distribution of unlabeled data.

\appendix

\section{Kernel density estimation review}
\label{app:kd}

We give a brief review of kernel density estimation, to justify the claims in Remark \ref{rem:iid}. The results are standard in the density estimation literature (see \cite{silverman2018density}), though not perhaps in the exact form we need, so we include them for completeness. 

We first state a preliminary proposition that is used in both of the following sections
\begin{proposition}\label{prop:taylor}
Suppose $\Phi$ satisfies \eqref{eq:phi} and $q:\R^d\to \R$ is $C^2$. Then for any $h>0$
\begin{equation}\label{eq:taylorq}
\frac{1}{h^d}\int_{\R^d}\Phi\left( \frac{|y-x|}{h} \right)q(y) \, dy =C_\Phi q(x) + O(h^2),
\end{equation}
where $C_\Phi := \int_{B(0,2)}\Phi(|z|) \, dz$.
\end{proposition}
\begin{proof}
Make the change of variables $z = (y-x)/h$ so that $dy = h^d dz$. Then
\begin{align*}
\frac{1}{h^d}\int_{\R^d}\Phi\left( \frac{|y-x|}{h} \right)q(y) \, dy &=\int_{\R^d}\Phi(|z|) q(x + zh)\, dz\\
&=\int_{B(0,2)}\Phi(|z|)\left( q(x) + h\nabla q(x)\cdot z + O(h^2) \right)\, dz\\
&= C_{\Phi}q(x) + O(h^2),
\end{align*}
since the $O(h)$ term is odd. This completes the proof.
\end{proof}

\subsection{The i.i.d.~case}
\label{app:iid}

For the \emph{i.i.d.}~case, our main tool is Bernstein's inequality. For $Y_1,\dots,Y_n$ \emph{i.i.d.}~with mean $\mu=\E[Y_1]$ and variance $\sigma^2 = \E((Y_1-\E[Y_1])^2)$, if $|Y_i|\leq M$ almost surely for all $i$ then Bernstein's inequality \cite{boucheron2013concentration} states that for any $t>0$
\begin{equation}\label{eq:bernstein}
\P\left( \left| \frac{1}{n}\sum_{i=1}^n Y_i - \mu \right|> t\right)\leq 2\exp\left( -\frac{nt^2}{2\sigma^2 + 4Mt/3} \right).
\end{equation}

Let $X_1,\dots,X_n$ be \emph{i.i.d.}~random variables on $\R^d$ with $C^2$ density $\rho$. For simplicity of presentation, we assume $\rho$ is compactly supported, and so $\rho$ is, in particular, bounded. For fixed $y\in \R^d$, the  normalized degree \eqref{eq:deg} (or kernel density estimator) is
\begin{equation}\label{eq:knn}
d_n(y) = \frac{1}{nh^d}\sum_{i=1}^n \Phi\left( \frac{|X_i-y|}{h} \right).
\end{equation}
Here, we apply Bernstein's inequality with $Y_i = \Phi\left( \frac{|X_i-y|}{h} \right)$, and so
\[\mu = \int_{\R^d}\Phi\left( \frac{|x-y|}{h} \right)\rho(x)\, dx\]
and 
\[\sigma^2 \leq \int_{\R^d}\Phi\left( \frac{|x-y|}{h} \right)^2\rho(x)\, dx \leq C\int_{B(x,2h)} \leq C h^d,\]
due to the assumption \eqref{eq:phi}. Applying Bernstein's inequality with $t=\lambda h^{d+1}$ yields
\[\P\left( \left| d_n(y) - h^{-d}\mu \right|> \lambda h\right)\leq 2\exp\left( -cnh^{d+2}\lambda^2 \right).\]
for all $0 \leq \lambda \leq h^{-1}$. By Proposition \ref{prop:taylor} we have
\[h^{-d}\mu = \frac{1}{h^d}\int_{\R^d}\Phi\left( \frac{|x-y|}{h} \right)\rho(x)\, dx=C_\Phi \rho(y) + O(h^2).\]
 Thus, for $f(y)=C_\Phi \rho(y)$ we have
\[\P\left( \left| d_n(y) - f(y) \right|> \lambda h + Ch^2\right)\leq 2\exp\left( -cnh^{d+2}\lambda^2 \right).\]
Hence \eqref{eq:Rn} holds almost surely, that is
\[\lim_{n\to \infty} \frac{|d_n(y) - f(y)|}{h_n}= 0,\]
provided that $h_n\to 0$ as $n\to \infty$ so that $\lim_{n\to \infty}nh_n^{d+2} = \infty$.  

We note that in the application of the result \eqref{eq:Rn}, $y$ is also random ($y\sim X_j$ for some $j$). To handle this, we first condition on $X_j$ and apply the argument above with $n$ replaced by $n-1$. Finally, to control $\max_{x\in \X_n}|d_n(x)-f(x)|$, as in \eqref{eq:Rn}, we union bound over all $n$ random variables in $\X_n$, which produces the extra $\log(n)$ factor in \eqref{eq:hs_inf2}.

\section{The non-\emph{i.i.d.}~case}
\label{app:noniid}

For the non-\emph{i.i.d.}~case, we use Bernstein inequality for $U$-statistics, which we recall now. Let $Y_1,\dots,Y_m$ be a sequence of $m$ \emph{i.i.d}~random variables on $\R^d$ and let $\tau:\R^d\times \R^d\to \R^d$ be a measurable function. The second order $U$-statistic is
\begin{equation}\label{eq:Ustatistic}
U_n =\frac{1}{n} \sum_{i\neq j} \tau(Y_i,Y_j), 
\end{equation}
where $n=m(m-1) = 2\binom{m}{2}$. Let $\mu = \E[\tau(Y_1,Y_2)]$ and $\sigma^2 = \text{Var}(\tau(Y_1,Y_2))$. The Bernstein inequality for $U$-statistics\cite{boucheron2013concentration} states that for all $t>0$
\begin{equation}\label{eq:bernsteinU}
\P(|U_n - \mu | \geq t ) \leq 2\exp\left( \frac{-\lfloor \frac{m}{2}\rfloor t^2}{2\sigma^2 + \frac{2}{3}\|h\|_\infty t} \right).
\end{equation}

We now describe our non-\emph{i.i.d.}~model. We assume the \emph{i.i.d.}~random variables $Y_1,\dots,Y_m$ have a $C^2$ density $\rho$ with compact support in $\R^d$. Let $\tau:\R^d \times \R^d \to \R^d$ be a smooth function with bounded first and second derivatives for which
\[x\mapsto \tau(x,y) \text{ and } y \mapsto \tau(x,y)\]
are invertible for all $(x,y)\in \R^d\times \R^d$. We also assume the Jacobians are bounded; that is, assume there exists $\theta>0$ such that
\begin{equation}\label{eq:jacobians}
|D_x\tau(x,y)|\geq \theta \text{ and } |D_y \tau(x,y)|\geq \theta
\end{equation}
for all $(x,y)\in \R^d\times \R^d$, where $|X|=|\det(X)|$. An example of such a $\tau$ is $\tau(x,y)=x+y$. Finally, our dataset is defined as 
\begin{equation}\label{eq:datanoniid}
X_n = \{\tau(Y_i,Y_j)\}_{i\neq j}. \ \ \ (1\leq i,j\leq m)
\end{equation}
 The dataset $X_n$ is a collection of $n$ identically distributed, but not independent, random variables.

Fix $z\in R^d$, and consider the degree  (kernel density estimator) \eqref{eq:deg} given by
\begin{equation}
d_n(z) = \frac{1}{nh^d}\sum_{x\in X_n}\Phi\left( \frac{|x-z|}{h} \right).
\end{equation}
We note this can be expressed as
\begin{equation}\label{eq:kdest}
d_n(z) = \frac{1}{nh^d}\sum_{i\neq j}\Phi\left( \frac{|\tau(Y_i,Y_j)-z|}{h} \right).
\end{equation}
We apply Bernstein's inequality for $U$-statistics with 
\[h(x,y) = \Phi\left( \frac{|\tau(x,y)-z|}{h} \right).\]
Here, we have
\[\mu = \int_{\R^d}\int_{\R^d} \Phi\left( \frac{|\tau(x,y)-z|}{h} \right)\rho(x) \rho(y)\, dx  dy,\]
and
\begin{align*}
\sigma^2&\leq\int_{\R^d}\int_{\R^d} \Phi\left( \frac{|\tau(x,y)-z|}{h} \right)^2\rho(x) \rho(y)\, dx  dy\\
&\leq C\int_{\R^d}\int_{\R^d} \Phi\left( \frac{|\tau(x,y)-z|}{h} \right)^2\, dx \rho(y) dy\\
&= C\int_{\R^d}\int_{\R^d} \Phi\left( \frac{|w-z|}{h} \right)^2 |D_x\tau(x,y)|^{-1}\, dw \rho(y) dy\\
&\leq C\theta^{-1}\int_{\R^d}\int_{B(z,h)} \Phi\left( \frac{|w-z|}{h} \right)^2\, dw \rho(y) dy\\
&\leq C\theta^{-1}\int_{\R^d}\int_{\R^d} \Phi\left( \frac{|w-z|}{h} \right)^2\, dw\rho(y)  dy\\
&\leq C\theta^{-1}h^d\int_{\R^d}\rho(y)  dy = C\theta^{-1}h^d.
\end{align*}
We will absorb $\theta^{-1}$ into $C$ from now on. Applying Bernstein \eqref{eq:bernsteinU} with $t = \lambda h^{d+1}$ we have
\[\P(|d_n(z) - h^{-d}\mu | \geq \lambda h ) \leq 2\exp\left( -c\sqrt{n} h^{d+2} \lambda^2\right).\]
for $0 < \lambda \leq h^{-1}$.
We now compute
\begin{align*}
h^{-d}\mu&=\frac{1}{h^d}\int_{\R^d}\int_{\R^d}\Phi\left( \frac{|\tau(x,y)-z|}{h} \right)\rho(x) \rho(y)\, dx  dy,\\
&=\int_{\R^d}\left[\frac{1}{h^d}\int_{B(0,2)}\Phi\left( \frac{|w-z|}{h} \right)\rho(\psi_y(w)) |D_x \tau(\psi_y(w),y)|^{-1} \, dw  \right]\rho(y) dy,
\end{align*}
where $\psi_y$ is the inverse of $x\mapsto \tau(x,y)$, that is $w = \tau(\psi_y(w),y)$. Applying Proposition \ref{prop:taylor} with 
\[q(w) = \rho(\psi_y(w)) |D_x \tau(\psi_y(w),y)|^{-1}\]
we have
\begin{align*}
h^{-d}\mu&=\int_{\R^d}\left(C_\Phi\rho(\psi_y(z)) |D_x \tau(\psi_y(z),y)|^{-1} + O(h^2)  \right)\rho(y) dy\\
&=C_\Phi\int_{\R^d}\rho(\psi_y(z)) |D_x \tau(\psi_y(z),y)|^{-1} \rho(y) dy + O(h^2).
\end{align*}
Setting
\[f(z) = C_\Phi\int_{\R^d}\rho(\psi_y(z)) |D_x \tau(\psi_y(z),y)|^{-1} \rho(y) dy \]
we have
\begin{equation}\label{eq:cmea}
\P(|d_n(z) - f(z)| \geq \lambda h + Ch^2 ) \leq 2\exp\left( -c\sqrt{n} h^{d+2} \lambda^2\right).
\end{equation}
Hence \eqref{eq:Rn} holds almost surely, that is
\[\lim_{n\to \infty} \frac{|d_n(z) - f(z)|}{h_n}= 0,\]
provided that $h_n\to 0$ as $n\to \infty$ so that $\lim_{n\to \infty}\sqrt{n}h_n^{d+2} = \infty$. 

As before, in the application \eqref{eq:Rn} $z$ is actually random, and $z\sim \tau(Y_\ell,Y_k)$. To handle this, we condition on both $Y_\ell$ and $Y_k$, and omit all dependent terms from the sum defining $d_n(z)$ in \eqref{eq:kdest}. There are $O(\sqrt{n})$ such terms, so we introduce an error of size $O(\frac{\sqrt{n}}{nh^d})$. Since we are assuming $\lim_{n\to \infty}\sqrt{n}h_n^{d+2} = \infty$, we have $\sqrt{n} \gg h_n^{-(d+2)}$ and so
\[\frac{\sqrt{n}}{nh_n^d} = \frac{1}{\sqrt{n}h_n^{d}} \ll h_n^2.\]
Hence, the omitted terms can be absorbed into the $O(h_n^2)$ error term in \eqref{eq:cmea}. 

Finally, to control $\max_{z\in X_n}|d_n(z)-f(z)|$, as in \eqref{eq:Rn}, we union bound over all $n$ random variables in $X_n$, which produces the extra $\log(n)$ factor present in \eqref{eq:hs_inf4}.

\end{document}